\documentclass[reqno,a4paper,11pt]{amsart}

\usepackage{amsfonts,amsmath,amsthm,amssymb,amscd}
\usepackage{amsthm}
\usepackage{comment}

\usepackage{chngcntr}
\usepackage{nccmath}

\usepackage[mathscr]{euscript}

\usepackage{frcursive}
\usepackage[T1]{fontenc}  

\usepackage{helvet}

\usepackage{anyfontsize}

\usepackage{epsf}

\usepackage{graphicx,psfrag}
\usepackage{tabularx}
\usepackage{multicol}
\usepackage{color}

\usepackage{mathtools}
\DeclarePairedDelimiter{\ceil}{\lceil}{\rceil}

\usepackage{scalerel,stackengine}
\stackMath
\newcommand\reallywidehat[1]{%
\savestack{\tmpbox}{\stretchto{%
  \scaleto{%
    \scalerel*[\widthof{\ensuremath{#1}}]{\kern-.6pt\bigwedge\kern-.6pt}%
    {\rule[-\textheight/2]{1ex}{\textheight}}
  }{\textheight}%
}{0.5ex}}%
\stackon[1pt]{#1}{\tmpbox}%
}

\newcommand{\verteq}{\rotatebox{90}{$\,=$}}

\font\sans=cmss12

\font\sc=cmcsc10 at 12truept 
\font\tsc=cmcsc8 at 9truept

\def \SaK{\text{\sans K}}

\def \ScptA{\mathscr{A}}
\def \ScptB{\mathscr{B}}
\def \ScptC{\mathscr{C}}

\def \ScptG{\mathscr{G}}

\def \Ak{\mathscr{A}}

\def \Gk{\mathscr{G}}
\def \Hk{\mathscr{H}}

\def \Jk{\mathscr{J}}

\def \Lk{\mathscr{L}}
\def \Mk{\mathscr{M}}

\def \Sk{\mathscr{S}}

\def \Uk{\mathscr{U}}

\def \Xk{\mathscr{X}}

\def \SamS{\text{{\fontfamily{phv}\selectfont{S}}}}

\def \mG{\text{\rm{\bf{G}}}}

\def \mH{\text{\rm{\bf{H}}}}

\def \mS{\text{\rm{\bf{S}}}}

\def \mU{\text{\rm{\bf{U}}}}

\def \mk{k}
\def \mkun{k^{un}}

\def \myAff{\text{\rm{Aff}}}

\def \Ad{\text{\rm{Ad}}}

\def \myAd{\text{\rm{Ad}}}
\def \myBall{\text{\rm{Ball}}}
\def \myblue{black}

\def \mychild{\text{\rm{c}}}
\def \my,{\hskip 0.02in}

\def \mydistD{\text{\rm{$\mathbb D$}}}

\def \myFacet{\text{\rm{Facet}}}

\def \myGal{\text{\rm{Gal}}}
\def \mygrad{\text{\rm{grad}}}
\def \myht{\text{\rm{ht}}}
\def \mycInd{\text{\rm{c-Ind}}}
\def \myInd{\text{\rm{Ind}}}

\def \myId{\text{\rm{Id}}}

\def \myint{\text{\rm{int}}}
\def \myOrb{\text{\rm{Orb}}}
\def \myparent{\text{\rm{p}}}

\def \myrank{\text{\rm{rank}}}

\def \mysph{\text{\rm{Sph}}}
\def \myTau{\text{\rm{T}}}
\def \myun{\text{\rm{un}}}

\def \myVsph{\text{\rm{VSph}}}

\def \Cic{{C^{\infty}_{c}}}

\def \Hom{\text{\rm{Hom}}}

\def \Lie{\text{\rm{Lie}}}
\def \meas{\text{\rm{meas}}}
\def \SL{\text{\rm{SL}}}

\def \myAd{\text{\rm{Ad}}}

\def \mysubsubsection{{\hskip 0.0in}}

\def \myInd#1#2{\text{\rm{Ind}}{\hskip 0.005in}^{#1}_{#2}}
\def \mycInd#1#2{\text{\rm{c-Ind}}{\hskip 0.005in}^{#1}_{#2}}

\def \mystar{\text{\rm{Star}}}

\def \bC{{\mathbb C}}

\def \bF{{\mathbb F}}

\def \bN{{\mathbb N}}

\def \bR{{\mathbb R}}

\def \bZ{{\mathbb Z}}

\def \fkg{{\mathfrak g}}

\def \myFFsl2{{\mathfrak s}{\mathfrak l}  (2,{\mathbb F}_{q})}
\def \myFFGL2{{\text{\rm{GL}}}  (2,{\mathbb F}_{q})}
\def \myFFSL2{{\text{\rm{SL}}}  (2,{\mathbb F}_{q})}

\def \myauthor{Allen Moy and Gordan Savin}

\catcode`\@=11

\@addtoreset{equation}{section} \@addtoreset{equation}{subsection}
\def\theequation{\ifnum\value{subsection}>0\relax
\thesubsection.\arabic{equation}\relax
\else\ifnum\value{section}>0\relax
\thesection.\arabic{equation}\relax \else\arabic{equation}\fi\fi}

\newtheorem{thm}[equation]{Theorem}
\newtheorem{lemma}[equation]{Lemma}
\newtheorem{prop}[equation]{Proposition}

\newtheorem*{thm*}{Theorem}
\newtheorem*{prop*}{Proposition}

\newtheorem{cor}[equation]{Corollary}

\newtheorem{defn}[equation]{Definition}

\newcommand \reBna{B}

\newcommand \reBCM{BCM}

\newcommand \reBD{BD}
\newcommand \reBS{BS}
\newcommand \reBKV{BKV}

\newcommand \reBTa{BTa}
\newcommand \reBTb{BTb}

\newcommand \reHCb{HC}

\newcommand \reKea{K}

\newcommand \reMS{MS}

\newcommand \reMPa{MPa}
\newcommand \reMPb{MPb}

\newcommand \reSSa{SS}

\newcommand \reTa{T}

\begin{document}
 
{\large{
 


\vskip 0.30in

\title[Euler-Poincar{\'{e}} formulae for positive depth Bernstein projectors]  
{Euler-Poincar{\'{e}} formulae for positive depth \\  Bernstein projectors}
\markboth{\myauthor}{Euler-Poincar{\'{e}} formulas for positive depth Bernstein projectors}

\author{\myauthor}

\thanks{The first author is partly supported by Hong Kong Research Grants Council grant CERG  {\#}16301915 and {\#}160301718.}

\thanks{The second author is partly supported by NSF grant DMS-1901745.}



\begin{abstract}

Work of Bezrukavnikov--Kazhdan--Varshavsky uses an equivariant system of trivial idempotents of Moy--Prasad groups to obtain an Euler--Poincar{\'{e}} formula for the r-depth Bernstein projector.  Barbasch--Ciubotaru--Moy use depth-zero cuspidal representations of parahoric subgroups to decompose the Euler--Poincar{\'{e}} presentation of the depth-zero projector.  For positive depth $r$, we establish a decomposition of the Euler--Poincar{\'{e}} presentation of the r-depth Bernstein projector based on a notion of associate classes of cuspidal pairs for Moy--Prasad quotients.  We apply these new Euler--Poincar{\'{e}} presentations to the obtain decompositions of the resolutions of Schneider--Stuhler and Bestvina--Savin.

\end{abstract}

 
\vskip 0.30in 
\maketitle 
 

\section{Introduction}\label{intro-ms}

\medskip

Suppose $\mk$ is a non-archimedean local field and $\mG$ is a reductive group defined over $\mk$.  Let $\Gk = \mG ( \mk )$ denote the group of $\mk$-rational points.  Let $\Omega (\Gk )$ denote the category of smooth (complex) representations of $\Gk$.  Bernstein (see [\reBna]) used parabolic induction to define natural full subcategories $\Omega ( \{ ( \Mk , \sigma ) \} )$ of $\Omega (\Gk )$ indexed by equivalence classes $\{ ( \Mk , \sigma ) \}$ of cuspidal data consisting of a Levi subgroup $\Mk$ (of $\Gk$) and an irreducible cuspidal representation of $\Mk$.  The cuspidal data equivalence is $(\Mk , \sigma ) \, \sim \, ({\Mk}' , \sigma ' )$, if  (i) there exists $g \in \Gk$, and (ii) there exists an unramified character $\chi$ of $\Mk$  so that ${\Mk}' = \myAd (g) ({\Mk})$ and $\sigma '$ and $(\sigma  \otimes \chi ) \circ \myAd (g)$ are equivalent.   These subcategories are called Bernstein components.   Furthermore (see [\reBD]) , to a Bernstein component $\Omega ( \{ ( {\Mk} , \sigma ) \} )$, there is an unique essentially compact $\Gk$-invariant distribution $P_{\Omega ( \{ ( {\Mk}  , \sigma ) \} )}$, i.e., an element of the Bernstein center,  with the property that: \quad
\smallskip
\begin{itemize}
\item[(i)] \ $P_{\Omega ( \{ ( {\Mk}  , \sigma ) \} )}$ is idempotent, \ 
\smallskip
\item[(ii)] \ for any smooth representation $(\pi ,V_{\pi} )$, the subrepresentation $\pi (P_{\Omega ( \{ ( {\Mk}  , \sigma ) \} )} ) (V_{\pi})$ lies in $\Omega ( \{ ( {\Mk} , \sigma ) \} )$, and \ 
\smallskip
\item[(iii)] \ if $(\pi , v_{\pi})$ is a smooth representation in $\Omega ( \{ (M , \sigma ) \} )$, then $\pi (P_{\Omega ( \{ ( {\Mk} , \sigma ) \} )} ) = \myId_{V_{\pi}}$.
\end{itemize}

\noindent The nomenclature Bernstein projector is used for such a distribution.

\medskip

As a consequence of the work [{\reMPa},{\reMPb}] of Moy-Prasad on unrefined minimal $\SaK$-types, to any irreducible representation $(\pi , V_{\pi})$ of $\Gk$, there is a nonnegative rational number $\rho (\pi )$ (called the depth of $\pi$).  It can be characterized as the smallest rational number $r$ so that exists $x \in \ScptB (\Gk )$ so that $V^{\Gk_{x,r^{+}}}_{\pi} \neq \{ 0 \}$ (and the action of $\Gk_{x,r}/\Gk_{x,r^{+}}$ on $V^{\Gk_{x,r^{+}}}_{\pi} \neq \{ 0 \}$ can be shown to be nondegenerate).   The depth is the same for all irreducible representations in a Bernstein component $\Omega = \Omega ( \{ (M , \sigma ) \} )$. We use the notation $\rho (\Omega )$ to denote this rational number.   The $r$-th depth projector $P_{\le r }$ is defined to be (the finite sum): 

\begin{equation}\label{depth-ler-projector-intro-ms}
P_{\le r} \ : = \ {\underset { \rho (\Omega ( \{ ( {\Mk} , \sigma ) \} ) ) \le r} \sum }  P_{\Omega ( \{ ( {\Mk} , \sigma ) \} )}  \ .
\end{equation}

If $N \in \bN_{+}$, and $r \in {\frac{1}{N}} \bN$,  Bezrukavnikov--Kazhdan--Varshavsky [{\reBKV}] replace the usual simplicial structure on the Bruhat--Tits building $\ScptB = \ScptB (\Gk )$, with a refinement.  For ease of exposition, we assume $\Gk$ is split and therefore $\ScptB (\Gk )$ has a hyperspecial point.  The new structure, which we denote as $\ScptB_{N}$, is obtained by taking each affine apartment $\ScptA \subset \ScptB$ and refining the simplicial structure on $\ScptA$.  One selects a hyperspecial point as an origin and then scales the locus of affine root hyperplanes which define the original simplicial structure on $\ScptA$ by a factor of ${\frac{1}{N}}$.  The resulting refined simplicial structure on $\ScptA$ is denoted $\ScptA_{N}$ and is independent of the hyperspecial point selected.    An important feature of this refined simplicial structure is that the Moy-Prasad groups $\Gk_{x,r^{+}}$ are constant on the interior of any facet of $F \subset \ScptB_{N}$.  Denote this group as $\Gk_{F,r^{+}}$.   Fix a Haar measure on $\Gk$, and define the $\Cic (\Gk )$ idempotent:
\begin{equation}\label{idempotent-intro-ms}
e^{r}_{F} \ := \ {\frac{1}{\meas (\Gk_{F,r^{+}} )}} \ 1_{\Gk_{F,r^{+}}} \ . 
\end{equation}

\smallskip
\noindent Inspired by work of Meyer--Solleveld in [{\reMS}], Bezrukavnikov--Kazhdan--Varshavsky use $\ScptB_{N}$ to give a presentation of the distribution $P_{\le r}$ as the Euler--Poincar{\'{e}} sum:

\smallskip

\begin{equation}\label{bkv-intro-ms}
P_{\le r} \ = \ {\underset {F \subset \ScptB_{N}} \sum } \ (-1)^{\dim (F) } \, e^{r}_{F} \ .
\end{equation}

\smallskip

A natural question to ask is whether there are similar Euler-Poincar{\'{e}} presentations for other linear combinations of Bernstein projectors.  An extreme case would be to ask if there is an Euler-Poincar{\'{e}} presentation for an individual projector $P_{\Omega ( \{ ( {\Mk} , \sigma ) \} )}$.  Evidence of the latter is the work in [{\reBCM}].  We recall the building $\ScptB (\Mk )$ of a Levi subgroup $\Mk$ is the union of the apartments $\ScptA (\Sk ) \subset \ScptB (\Gk )$ as $\Sk$ runs over the maximal split tori in $\Mk$, and that for $x \in \ScptB (\Mk )$ we have $\Mk_{x,r} = (\Gk_{x,r} \cap \Mk)$.   When the class $\{ ( {\Mk} , \sigma ) \}$ has depth zero, it is a consequence of work in [{\reMPb}] that there is a facet $F \subset \ScptB$ so that 
\smallskip
\begin{itemize}
\item[(i)] \ $F \subset \ScptB ( \Mk ) \subset \ScptB (\Gk )$  (hence that $\Gk_{F,0}/\Gk_{F,0^{+}} = \Mk_{F,0} / \Mk_{F,0^{+}}$).
\smallskip
\item[(ii)] a cuspidal representation $\tau$ of the finite field group $\Mk_{F,0} / \Mk_{F,0^{+}}$ so that the cuspidal representation $\sigma$ contains $\tau$ (inflated to $\Mk_{F,0}$).  
\end{itemize}
\smallskip
\noindent Such a pair $(F, \tau )$ is called a cuspidal pair.  As indicated in [{\reMPb}], uniqueness of the pair $(F, \tau )$ is up to the equivalence of associate pairs, which we now recall.  Two facets $F$ and $F'$ are associate if $\dim (F) = \dim (F')$ and there exists $g \in \Gk$ so that the convex facet closure $C(F,g.F')$ of $F$ and $g.F'$ also has dimension this common dimension.  We say the facets $F$ and $g.F'$ are aligned.  We define two cuspidal pairs $(F , \tau )$ and $(F' , \tau ')$ to be associate if they satisfy the following:  
\begin{itemize} 
\item[(i)] \ The facets $F$ and $F'$ are associate, e.g., assume $F$ and $g.F'$ are aligned.   The alignment of $F$ and $g.F$ yields a canonical identification of $\Gk_{F,0}/\Gk_{F,0^{+}}$ and  $\Gk_{g.F',0}/\Gk_{g.F',0^{+}}$.  
\smallskip
\item[(ii)] \ The element $g$ in part (i) can be taken so that  $\tau \circ \myAd (g)$ is to be equivalent to $\tau '$.
\end{itemize}

\smallskip

\noindent A cuspidal pair $(F,\tau )$ defines an associate class ${\mathcal C} = \{ (F,\tau ) \}$, and to the associate class ${\mathcal C}$ there is an $\Gk$-equivariant system of idempotents $e^{\my, \mathcal C}_{K}$ ($K$ a facet in $\ScptB (\Gk )$) so that
 
\begin{equation}\label{bcm-intro-ms}
P^{\my, \mathcal C}_{0} \ := \ {\underset {K \subset \ScptB} \sum } \ (-1)^{\dim (K) } \, e^{\my, \mathcal C}_{K} \ 
\end{equation} 

\noindent is an Euler--Poincar{\'{e}} presentation of the sum {\,}$\big( \, {\sum}' \, P_{\Omega ( \{ ( {\Lk} , \kappa ) \} )} \, \big)${\,} over the classes  $\{ ( {\Lk} , \kappa ) \}$ which contain a pair $(\Mk , {\theta})$ so that the restriction $\theta_{{|}\Mk_{F,0}}$ contains $\tau$.  We caution the double use of notation (which should be clear from context): \ the class $\{ ( {\Lk} , \kappa ) \}$ is that of a Levi subgroup $\Lk$ and a cuspidal representation $\kappa$ of $\Lk$, while class $\{ ( F , \tau ) \}$ is a facet $F$, and a cuspidal representation $\tau$ of $\Gk_{F,0}/\Gk_{F,0^{+}}$ inflated to $\Gk_{F,0}$.  

\medskip

A key result of [{\reBCM}] is the (orthogonal) decomposition:
\begin{equation}\label{hc-intro-ms}
P_{\le0} \ = \ {\underset 
{\text{\rm{\tiny $\begin{matrix} {\mathcal C}  \end{matrix}$}}} \sum} \  P^{\my, \mathcal C}_{0} \ .
\end{equation}
 
\noindent We view this decomposition of $P_{\le 0}$ as an embodiment of Harish-Chandra's philosophy of cusp forms [{\reHCb}] over finite field groups.  

\medskip

Our goal here is to show a partial analogue of \eqref {hc-intro-ms} for positive depth.  We assume $N \in \bN_{+}$ and $r \in {\frac{1}{N}}\bN_{+}$.  Let $\ScptB_{N}$ denote the Bruhat--Tits building with the refined simplicial structure mentioned above.  In addition to constancy of the groups $\Gk_{x,r^{+}}$ on the interior of a facet $F$, the groups $\Gk_{x,r}$ are constant too (we use the notations $\Gk_{F,r^{+}}$ and $\Gk_{F,r}$ for these groups), and if $E$ is a subfacet of $F$, then $\Gk_{E,r} \supset \Gk_{F,r}  \supset \Gk_{F,r^{+}} \supset \Gk_{E,r^{+}}$.  The condition $r > 0$ has the consequence that the quotient $\Gk_{F,r} / \Gk_{F,r^{+}}$ is commutative.

\bigskip


Let $F\supset E$ be a pair of facets in $\ScptB_{N}$. A facet $\bar F\supset E$ is called opposite to $F$ if there exists an apartment $\ScptA (\Sk )$ that contains $F$ and $\bar F$ is the reflection of $F$ about the affine subspace of $\ScptA (\Sk )$ generated by $E$. 
Two opposite facets $F, \, F'$ yield (see \eqref{iwahori-b-prelim-ms} in Proposition \ref{iwahori-a-prelim-ms}) 
 an Iwahori decomposition  of $\Gk_{E,r}/\Gk_{E,r^{+}}$ as:
\begin{equation}\label{iwahori-decompo-intro-ms}
\Gk_{E,r}/\Gk_{E,r^{+}} \ = \ \Gk_{\bar F,r^{+}}/\Gk_{E,r^{+}} \ \oplus \ \Gk_{F,r}/\Gk_{F,r^{+}} \ \oplus \ \Gk_{F,r^{+}}/\Gk_{E,r^{+}} \ .
\end{equation}

\noindent  Since $F$ and $\bar F$ are aligned, note that we have a canonical identification


\begin{equation}\label{opposite-intro-ms}
\Gk_{F,r}/\Gk_{F,r^{+}} \ =  \ \Gk_{\bar F,r}/\Gk_{\bar F,r^{+}} \ .
\end{equation}


\medskip

Let $E$ be a facet of $\ScptB_{N}$.  
In homage to the usual notions of parabolic induction and restriction, we define a character $\chi$ of $\Gk_{E,r}/\Gk_{E,r^{+}}$ to be cuspidal, if for any facet $F \supsetneq E$, the restriction of $\chi$ to the summand $\Gk_{F,r^{+}}/\Gk_{E,r^{+}}$ in \eqref{iwahori-decompo-intro-ms} is non-trivial.  
Propositions \ref{property-pontryagin-ms} and \ref{uniqueness-pontryagin-ms} say to any character $\chi$ of $\Gk_{E,r}/\Gk_{E,r^{+}}$, there exists a facet $F$ containing $E$ and a cuspidal character $\phi$ of $\Gk_{F,r}/\Gk_{F,r^{+}}$ so that the inflation of $\phi$ to $\Gk_{F,r}$ equals $\chi_{| \Gk_{F,r}}$; or, equivalently, that $\chi$ is contained in the induced representation  $\myInd{\Gk_{E,r}}{\Gk_{F,r}} \, \phi$.

\medskip

We define a cuspidal pair $(E, \chi )$ as consisting of a facet $E\subset \ScptB_{N}$ and a cuspidal character of $\Gk_{E,r} / \Gk_{Er^{+}}$.   For this definition, there is a precise analogue of the depth zero equivalence relation of associate cuspidal pairs.  As in the depth zero situation, we call the equivalence class of a cuspidal pair $(F, \chi)$, its associate class ${\mathcal C} = \{ {\,} (E', {\chi}' ) {\,} | {\,} (E', {\chi}' ) \sim  (E, {\chi} ) {\,} \}$.  To the associate class ${\mathcal C}$, in a fashion analogous to the depth zero situation, there is an $\Gk$-equivariant system of idempotents $e^{\my, \mathcal C}_{J}$ ($J$ a facet in $\ScptB_{N}$).  The support of the equivariant system is on the facets which are contained in a facet of the associate class.  Our main Theorem (Theorem \ref{main-convolve-ms}) is that the Euler-Poincar{\'{e}} sum 
\begin{equation}\label{ep-sum-intro-ms}
P^{\my, \mathcal C} \ := \ {\underset {J \subset \ScptB_{N}} \sum } \ (-1)^{\dim (J) } \, e^{\my, \mathcal C}_{J} \ 
\end{equation}

\noindent is a distribution in the Bernstein center and idempotent (hence a finite sum of Bernstein projectors), and the depth $r$ project $P_{\le r}$ has the decomposition 
\begin{equation}\label{ms-intro-ms}
P_{\le r} \ = \ {\underset {\mathcal C} \sum} \ P^{\my, \mathcal C} \ .
\end{equation} 

\noindent We note that there is an associate class ${\mathcal N}$ associated to cuspidal pairs of the type $( C , \chi_{\text{\rm{triv}}} )$, where $C$ is a chamber in $\ScptB_{N}$, and that (in the obvious notation) $P_{< r} = P^{\my, {\mathcal N}}$.  Let $r_{0}$ to be the largest depth of a representation of $\Gk$ which is strictly less than $r$.  Then $P_{<r} = P_{\le r_{0}}$.  The Euler--Poincar{\'{e}} presentation $P^{\my, {\mathcal N}}$ of $P_{<r}$, and the Euler--Poincar{\'{e}} presentation \eqref{bkv-intro-ms} of $P_{\le r_{0}}$ are two presentations of the same distribution.  It would be interesting to have a direct proof of this connection.

\medskip

For respectively the Bruhat-Tits building $\ScptB$ and the refined simplicial structure $\ScptB_{N}$, the idempotents $e^{r}_{J}$ appear in the work of Schneider--Stuhler and  Bestvina--Savin on resolutions of a smooth representation $(\pi , V_{\pi})$.  Indeed, for $r \in \bN$, Schneider--Stuhler independently discovered the groups $\Gk_{F,r}$ and used the notation $U^{r}_{F}$ for these groups.  Both  Schneider--Stuhler and Bestvina--Savin assemble the vector spaces $\pi (e^{r}_{J}) (V_{\pi})$ into a smooth representation of $\Gk$ 
\begin{equation}\label{resolution-b-intro-ms}
W^{r,k}_{\pi} \ = \ {\underset {\text{\rm{\tiny $\begin{matrix} {J \subset {\ScptB_{N}} } \\ {\dim ( J ) = k} \end{matrix}$ }}} \bigoplus } (J , \pi (e^{r}_{J}) (V_{\pi}))
\end{equation}
\noindent which is projective.  They define natural  boundary maps $W^{r,k}  \xrightarrow{ \ \ \partial \ \ } W^{r,(k-1)}$ and an augmentation map $W^{r,0} \longrightarrow V_{\pi}$ so that if $V_{\pi}$ is generated by the image of the augmentation map, then the resulting complex is a resolution of $V_{\pi}$.   

\medskip

For an associate class ${\mathcal C} = \{ (F , \chi ) \}$, a key property of the idempotent distribution $P^{\my, \mathcal C}$ (see Theorem \ref{main-convolve-ms} part (iii)) is that 
$$
P^{\my, \mathcal C} \ \star \ e^{r}_{J} \ = \ e^{\my, \mathcal C}_{J} \qquad {\text{\rm{for any facet $J \, \subset \, \ScptB_{N}$}}} \ . 
$$
A consequence of this is that when $P^{\my, \mathcal C}$ is applied to a resolution of Schneider--Stuhler or of Bestvina--Savin, we obtain
$$
\pi (P^{\my, \mathcal C}) \ \big( \, \pi (e^{r}_{J}) (V_{\pi}) \, \big) \ = \  \pi (e^{\my, \mathcal C}_{J}) (V_{\pi}) \ .
$$
\noindent Thus, if we define $W^{{\my, \mathcal C},k}$ as in \eqref{resolution-b-intro-ms} by replacing $\pi (e^{r}_{J}) (V_{\pi})$ with $\pi (e^{\my, \mathcal C}_{J}) (V_{\pi})$, we obtain a resolution of $\pi (P^{\my, \mathcal C})(V_{\pi})$.

\medskip

We outline our presentation of results.  In section \ref{prelim-ms}, we introduce notation and the simplicial refinement $\ScptB_{N}$ ($N \in \bN$) of the Bruhat--Tits building $\ScptB$.  If $E$ and $E'$ are facets of $\ScptB_{N}$, we introduce the set $C(E,F)$ which is the convex facet closure of $E$ and $F$.  The set $C(E,F)$ is used to defined the notion of aligned facets and associate facets.  It is essential in formulating the important convolution result Proposition \ref{convolution-prelim-ms}. 

\medskip

In section \ref{pontryagin-ms} we define cuspidal characters of the quotient groups $\Gk_{F,r}/\Gk_{F,r^{+}}$ (and their inflations to   $\Gk_{F,r}$), and establish that if $\chi$ is a character of $\Gk_{E,r}/\Gk_{E,r^{+}}$ then there is a facet $F \supseteq E$ and a cuspidal character $\phi$ of $\Gk_{F,r}/\Gk_{F,r^{+}}$, so that $\chi_{|\Gk_{F,r}}$ is the inflation of $\phi$.  Furthermore if the same is true for another cuspidal pair $(F',{\phi}')$, then it is associate to $(F,\phi )$.  We also establish useful convolution properties of cuspidal characters.  We end the section defining the $\Gk$-equivariant system of idempotents $e^{\my, \mathcal C}_{J}$ ($J$ a facet of $\ScptB_{N}$) associated to an associate class ${\mathcal C}$ of a cuspidal pair.

\medskip
  
In section \ref{ep-ms}, we fix a chamber $C \subset \ScptB_{N}$ and establish formulae for the Euler--Poincar{\'{e}} sums over certain subfacets of $C$.  In section \ref{keyprop-ms}, we fix a base chamber $C_{0} \subset \ScptB_{N}$ and  establish Proposition \ref{proposition-keyprop-ms} which is key to obtaining what we believe to be an elegant proof that the Euler--Poincar{\'{e}} sum \eqref{ep-sum-intro-ms} is in the Bernstein center.  In section \ref{convolve-ms},
we prove our main Theorem  \ref{main-convolve-ms} about the distributions $P^{\my, \mathcal C}$ of \eqref{ep-sum-intro-ms}.  In section \ref{resolutions-ms}, we apply the distributions $P^{\my, \mathcal C}$ to obtain decompositions of the resolutions (see [{\reSSa},{\reBS}]) of Schneider--Stuhler and Bestvina--Savin.
In section \ref{nonsplit}, we exposit the change necessary in the definition of virtual affine roots to extend the split arguments to nonsplit groups.

\medskip


\vskip 0.70in 
 

\section{Refined simplicial structure on a Bruhat--Tits building}\label{prelim-ms}

\bigskip

\subsection{Notation and Preliminaries} \quad 

\medskip

As in the introduction, we assume $\mk$ is a non-archimedean local field.  We denote by ${\mathfrak O}_{\mk}$, $\wp_{\mk}$, and $\bF_{q} = {\mathfrak O}_{\mk}/\wp_{\mk}$ respectively, the ring of integers, prime ideal, and residue field of $\mk$.   We use similar notation for a maximal unramified extension ${\mkun}/{\mk}$.   The residue field ${\mathfrak O}^{\myun}_{\mk}/\wp^{\myun}_{\mk}$ is an algebraic closure of $\bF_{q}$ and we denote it as ${\overline{\bF_{q}}}$. \ Let $\mG$ be a connected reductive linear algebraic group defined over $\mk$.   For convenience, we assume $\mG$ is $\mk$-split and quasisimple.   Set $\ell = \myrank (\mG)$.  If $\mH$ is a $\mk$-subgroup of $\mG$, we write $\Hk$ for the group $\mH (\mk )$ of $\mk$-rational points of $\mH$, e.g., $\Gk = \mG (\mk )$.   We set $\Gk^{\myun} := \mG (\mkun )$, and $\Gk := \mG (\mk )$.

\medskip

Let $\mS$ be a maximal $\mkun$-split torus of $\mG$ (so, by definition, $\Sk^{\myun} = \mS (\mkun )$), and let (see [{\reTa},{\reBTa},{\reBTb}]) $\Ak (\Sk^{\myun} )$ be the apartment associated to $\Sk^{\myun}  \subset \Gk^{\myun}$.  Let 
\begin{equation}\label{roots-prelim-ms}
\aligned
\Phi (\mS ) \ := \ &{\text{\rm{the set roots of $\mG$ with respect to $\mS$}}} \ , \\
\Psi (\mS ) \ := \ &\{ \ \alpha + k \ | \ \alpha \, \in \, (\Phi \cup \{ 0 \} ) \ , \ k \in \bZ \ \} \\ 
&\ {\text{\rm{the set of affine roots (with respect to $\mS$)}}} \ .  
\endaligned
\end{equation}

\noindent  When clear, we abbreviate these two sets to $\Phi$ and $\Psi$ (as well as $\Phi (\Sk )$ and $\Psi (\Sk )$) respectively.  To a root $\alpha \in \Phi$ (resp.~affine root $\psi \in \Psi$), let $\Uk_{\alpha}$ (resp.~$\Xk_{\psi}$) denote the attached root group (resp.~affine root group).

\medskip

When $\xi$ is a nonconstant affine function on $\Ak(\Sk^{\myun} )$, we define 
\begin{equation}\label{hyperplane-defn-prelim-ms}
H_\xi \ := \ {\text{\rm{the zero hyperplane of $\xi$}}} .
\end{equation}
\noindent We recall that an affine root hyperplane is a hyperplane $H_{\psi}$ attached to a nonconstant affine root $\psi \in \Psi$. \ Our hypotheses on $\mG$ ($\mk$-split, quasisimple) and $\mS$ (maximal $\mkun$-split torus) means the apartment $\Ak ( \mkun )$ is a simplicial complexes under the decomposition by the zero hyperplanes of the nonconstant affine roots.   The Bruhat-Tits building  $\ScptB (\Gk^{\myun} )$ of $\Gk^{\myun}$ is obtained by gluing apartments $\Ak ( \Sk^{\myun}  )$ ($\mS$ running over all $\mkun$-split tori) along (convex) sets of subsimplicies.  Thus, $\ScptB ( \Gk^{\myun} )$ is a simplicial complex and metric space, with a simplicial and isometrical action of  $\Gk^{\myun}$.    The Galois group ${\text{\rm{Gal}}}(\mkun / \mk )$ acts on $\Gk^{\myun}$  with fixed points $(\Gk^{\myun})^{{\text{\rm{Gal}}}(\mkun / \mk )}$ equal to $\Gk$. \ So, $\Gk$ acts on the Bruhat-Tits building:
\begin{equation}\label{building-prelim-ms}
\ScptB (\Gk) := \ScptB (\Gk^{\myun})^{{\text{\rm{Gal}}}(\mkun / \mk )} \ . 
\end{equation}
\noindent In fact, $\ScptB (\Gk )$ can be obtained by gluing apartments  $\Ak (\Sk^{\myun} )$ of maximal $\mk$-split tori $\mS$.  We conveniently denote such an apartment also as $\ScptA (\Sk )$.

\medskip

The group $\Gk^{\myun}$ (resp.~$\Gk$) acts transitively on the chambers, i.e., $\ell$-simplicies, of $\ScptB (\Gk^{\myun})$ (resp.~$\ScptB (\Gk)$).  The choice of a hyperspecial point $x_{0} \in \ScptA =  \ScptA (\Sk )$ corresponds to the choice of a Chevalley basis for the Lie algebra $\fkg$ of $\Gk$.   Such a choice gives an identification of $\ScptA (\Sk )$ with $\Hom (\mG_{m} , \mS ) \otimes_{\bZ} \bR$ whereby the point $x_{0}$ becomes $0$.

\bigskip

\subsection{Simplicial refinement by $N \in \bN_{+}$}  \quad 

\medskip

\mysubsubsection For {\,}$N \in \bN_{+}$, the identification of $\ScptA (\Sk )$ with $\Hom (\mG_{m} , \mS ) \otimes_{\bZ} \bR$ allows us to scale the simplicial structure on {\,}$\ScptA (\Sk )${\,} by a factor of {\,}$\frac{1}{N}$.  Each facet of $\ScptA (\Sk )$ is a union of facets of the new simplicial structure.  So, the new simplicial structure is a refinement of the original structure.  We use the notation $\ScptA (\Sk )_{N}$ to denote $\ScptA (\Sk )$ with this refined simplicial structure.

\medskip

\noindent Equivalently, $\ScptA (\Sk )_{N}$ can be described as the refined simplicial structure on $\ScptA (\Sk )$ arising from the zero hyperplanes of the set of virtual affine roots defined as:

\vskip -0.10in

\begin{equation}\label{virtualn-prelim-ms}
\Psi (\Sk )_{N} \, := \, \left\{ \
\begin{matrix}
\alpha + \ell \\
\alpha \in \Phi \ , \ \ell \in {\frac{1}{N}} \, \bZ
\end{matrix}
\ \right\} \ \ \supset \ \ \Psi ( \Sk ) \ .
\end{equation}

\vskip 0.10in

\noindent We shall sometimes refer to these hyperplanes as refined affine hyperplanes, and sometimes lapse into calling virtual affine roots refined affine roots.  The refined simplicial structure $\ScptA (\Sk )_{N} $ of each apartment $\ScptA (\Sk )$ yields a refined simplicial structure $\ScptB (\Gk )_{N}$ of $\ScptB (\Gk )$.

\bigskip

\subsection{Simplicial closure of two facets} \ 

\medskip

If $E, F  \, \subset \, \ScptB$ are facets, we define their 
{\color{\myblue}convex simplicial closure} is defined as:
\begin{equation}\label{c-prelim-ms}
\aligned
C(E,F) \ :&= \ {\text{\rm{smallest convex union of facets of ${\ScptB}$ }}} \\
&{\hskip 0.70in} {\text{\rm{containing $E$ and $F$}}} \\
&=\ {\underset {\tiny{\begin{matrix} {\text{\rm{apartment $\ScptA$,}}} \\
{\text{\rm{so that  $E , F \subset \ScptA$}}} \end{matrix}}} \bigcap } \ \ScptA \ \ . 
\endaligned
\end{equation}

\noindent Suppose $N \in \bN_{+}$.   The above can be extrapolated to refined facets $E, F \, \subset \, \ScptB_{N}$ as follows.
\begin{equation}\label{ccc-prelim-ms}
\aligned
C(E,F) \ :&= \ {\text{\rm{smallest convex union of facets of ${\ScptB}_{N}$}}} \\
&{\hskip 0.70in} {\text{\rm{containing $E$ and $F$.}}} \\
\endaligned
\end{equation}


\begin{defn}  Suppose $E$ and $F$ are facets of $\ScptB_{N}$. 
\begin{itemize}

\item[$\bullet$] \ Define facets $E$ and $F$ to be aligned if
\begin{equation}\label{aligned-prelim-ms}
{\text{\rm{$\dim(C(E,F)) = \dim (E) = \dim (F)$.{\hskip 0.70in} }}}
\end{equation}

\item[$\bullet$] \ Define a facet $E$ of $\ScptB_{N}$ is {\color{\myblue} subaligned} to a facet $F$ if:
\begin{equation}\label{subaligned-prelim-ms}
\aligned
{\text{\rm{(i)}}} &{\ \ }{\text{\rm{$\dim (E) \le \dim (F)$}}} \\
{\text{\rm{(ii)}}} &{\ \ }{\text{\rm{$\dim(C(E,F)) = \dim (F)$. {\hskip 1.45in} }}}
\endaligned
\end{equation} 
\noindent We use the term  {\color{\myblue} strictly subaligned} to mean $\dim (E) < \dim (F)$ and subaligned.  
\end{itemize}
\end{defn}

\noindent A consequence of facets $E$ and $F$ being aligned is that there is a canonical equality of $\Gk_{E,0}/\Gk_{E,0^{+}}$ and   $\Gk_{F,0}/\Gk_{F,0^{+}}$ and also of $\Gk_{E,r}/\Gk_{E,r^{+}}$ and   $\Gk_{F,r}/\Gk_{F,r^{+}}$.   If $E$ is subaligned to $F$, then we can view 
$\Gk_{F,0}/\Gk_{F,0^{+}}$ and $\Gk_{F,r}/\Gk_{F,r^{+}}$ as canonically inside $\Gk_{E,0}/\Gk_{E,0^{+}}$ and $\Gk_{E,r}/\Gk_{E,r^{+}}$ respectively.

\medskip

\noindent

\begin{lemma}\label{facet-prelim-ms} Given facets $E,F \subset \ScptB_{N}$, there are unique facets $D_{E} , D_{F} \subset C(E,F)$ so that:
\vskip 0.05in
\begin{itemize}
\item[(i)] \ $\dim (D_{E}) = \dim (D_{F}) = \dim ( C(E,F))$; thus, $D_{E}$ and $D_{F}$ are aligned (so both $\Gk_{D_{E},0}/\Gk_{D_{E},0^{+}} = \Gk_{D_{F},0}/\Gk_{D_{F},0^{+}}$ and $\Gk_{D_{E},r}/\Gk_{D_{E},r^{+}} = \Gk_{D_{F},r}/\Gk_{D_{F},r^{+}}$).
\medskip
\item[(ii)] \ $E \subset D_E$ and $F \subset D_{F}$; thus, the parahoric subgroup $\Gk_{D_{E},0}$ (resp.~$\Gk_{D_{F},0}$) is contained in $\Gk_{E,0}$ (resp.~$\Gk_{F,0}$). 
\end{itemize}
\end{lemma}

\medskip

\begin{proof} We recall that any apartment $\ScptA$ which contains the (refined) facets $E$ and $F$ also contains their convex closure $C(E,F)$.  Let $D_{E} \subset C(E,F)$ be a refined facet of maximal dimension $\dim (C(E,F))$ containing $E$.  In $\ScptA$, the  affine subspace $\myAff (E)$ contains $C(E,F)$.  Suppose $D'_{E} \subset C(E,F)$ is another refined facet containing $E$ of dimension $\dim (C(E,F))$.  If $D_{E} \neq D'_{E}$, there must be an virtual affine root $\psi$ so that \ (i)
the intersection $H_{\psi} \cap \myAff (C(E,F))$ is codimension one in $\myAff (C(E,F))$, and so divides it into two halfspaces, and \ (ii)  $D_{E}$ and $D'_{E}$ are in opposite halfspaces.  Obviously $C(E,F)$ must be in the closure of exactly one of the halfspaces which contradicts $D_{E}$ and $D'_{E}$ being subsets of 
$C(E,F)$.  Thus, $D_{E}$ is unique.  Similarly for $F$.

\end{proof}

\bigskip

\subsection{Iwahori factorizations} \

\medskip

We fix an apartment $\ScptA = \ScptA (\Sk )$ of $\ScptB = \ScptB (\Gk )$.  If $\alpha \in \Phi = \Phi (\Sk)$, let $\Uk_{\alpha}$ denote the $\alpha$-root group.  Similarly, if $\psi \in \Psi$, let $\Xk_{\psi}$ denote the affine root group indexed by $\psi$.  \  Fix $x \in \ScptA$ and $r > 0$.  For any $\alpha \in \Phi \cup \{ 0 \}$, set 
$$
{\text{\rm{$\psi_{\alpha , x ,r } \ := \ $ the smallest affine root $\psi$ with $\mygrad(\psi ) = \alpha$, and {\,}$\psi (x) \ge r$.}}}
$$

\noindent We recall for any choice $\Phi^{+}$ and $\Phi^{-} = - \Phi^{+}$ of positive and negative roots, that the group $\Gk_{x,r}$ has the Iwahori factorization:

\begin{prop} {\text{\rm{(Iwahori factorization)}}} \quad $\Gk_{x,r} \, = \, \Uk^{-}_{x,r} \, \Sk_{r} \, \Uk^{+}_{x,r}$, where
$$
\Uk^{-}_{x,r} \  = \ {\underset {\alpha \in \Phi^{-}} \prod } \Xk_{\psi_{ \alpha , x , r }} 
{\text{\rm{ \quad and \quad }}}
\Uk^{+}_{x,r} \  = \ {\underset {\alpha \in \Phi^{+}} \prod } \Xk_{\psi_{ \alpha ,  x , r }} \ .
$$
\end{prop}

\noindent If $\psi = \alpha + \ell$ is a virtual affine root, define the {\color{\myblue} {\it{(virtual) affine root group}}}  $\Xk_{\psi}$ as:
$$
\Xk_{\psi} \ := \ \Xk_{(\alpha + \ceil{\ell{\hskip 0.01in}})} \ .
$$
\noindent The (virtual) affine root groups satisfy analogous commutation relations as the affine root groups.  In the definition of the groups $\Gk_{x,r}$, we can replace affine root groups with virtual affine root groups, but we obviously do not get any different groups.

\begin{lemma}  \ Given $x \in \ScptA = \ScptA (\Sk)$, and $r>0$:
\smallskip
\begin{itemize}
\item[(i)] \ $\Gk_{x,r} \ = \ {\underset {\tiny{\begin{matrix} {\text{\rm{$\psi$ virtual affine}}} \\ {\psi (x) \ge r} \end{matrix}}} \prod } \Xk_{\psi}$.
\smallskip
\item[(ii)] For a choice of positive roots $\Phi^{+}$:
$$
\Uk^{-}_{x,r} = {\underset {\tiny{\begin{matrix} {\text{\rm{$\psi$ virtual affine}}} \\ \mygrad (\psi ) \in \Phi^{-} \\ {\psi (x) \ge r} \end{matrix}}} \prod } \Xk_{\psi} \qquad {\text{\rm{and}}} \qquad \Uk^{+}_{x,r} = {\underset {\tiny{\begin{matrix} {\text{\rm{$\psi$ virtual affine}}} \\ \mygrad (\psi ) \in \Phi^{+} \\ {\psi (x) \ge r} \end{matrix}}} \prod } \Xk_{\psi} \ .
$$ 
\end{itemize}
\end{lemma}

\begin{proof}  Elementary verification.
\end{proof}

\medskip

\begin{lemma}\label{bf-prelim-ms}  \quad Fix {\,}$N \in \bN_{+}$.  \ For {\,}$r \in {\frac{1}{N}}  \, \bN$, and a (refined) facet {\,}$F \subset \ScptB (\Gk)_{N}$:
\begin{itemize}
\item[(i)] The groups $\Gk_{x,r^{+}}$ and $\Gk_{x,r}$ are constant on the interior of any facet of {\,}$F$. 
\vskip 0.05in
\noindent{\hskip 0.03in}Let {\,}$\Gk_{F,r^{+}}$ and  {\,}$\Gk_{F,r}${\,}  denote these groups.
\smallskip
\item[(ii)] Suppose $E \subset F$ is a (refined) subfacet of $F$, then 
$$
\Gk_{E,r} \ \supset \ \Gk_{F,r} \ \supset \ \Gk_{F,r^{+}} \ \supset \ \Gk_{E,r^{+}} \ .
$$
\end{itemize}
\end{lemma}

\begin{proof} \ \ To prove (i), suppose $\psi$ is a virtual affine root.  If sufficies to show 
\begin{equation}\label{bf-a-prelim-ms}
\aligned
\forall \ x, \, y \, \in \, \myint (F) \quad : \quad \Xk_{\psi} \ \subset \ \Gk_{x,r^{+}} \quad &\iff \quad \Xk_{\psi} \ \subset \ \Gk_{y,r^{+}} \\
\quad \Xk_{\psi} \ \subset \ \Gk_{x,r} {\hskip 0.08in} \quad &\iff \quad \Xk_{\psi} \ \subset \ \Gk_{y,r} \ . \\
\endaligned
\end{equation}
This is obvious if $\psi$ is constant on $\myint (F)$.  \ If $\psi$ is not constant on $\myint (F)$, then there exists $j \in \bZ$ so that {\,}$\forall \ x \, \in \, \myint (F) \ :  {\frac{j}{N}} \, < \, \psi \, (x) \, < \, {\frac{j+1}{N}}$.  \ This is because otherwise the values of $\psi$, on $\myint (F)$,  would contain an open neighborhood of some  ${\frac{j}{N}}$, and hence the hyperplane $H_{\big( \psi - {\frac{j}{N}} \big)}$ would cut $F$ into two strictly smaller pieces, contradicting the assumption $F$ is a refined facet.   So, \eqref{bf-a-prelim-ms} is true, and the two assertions of the statement (i) hold.

\smallskip

\noindent Statement (ii) follows from (i).
\end{proof}

\smallskip

\noindent Statement (ii) of the Lemma for the value $r=0$ is the assertion the parahoric subgroup $\Gk_{E,0}$ contains the parahoric subgroup $\Gk_{F,0}$.    

\bigskip

For a facet $E \subset \ScptB (\Gk )_{N}$, define the {\color{\myblue} \it{virtual sphere}} {\,}${\myVsph}(E)$ of $E$ as the set:
\begin{equation}\label{virtual-sphere-prelim-ms} 
{\myVsph}(E) \ : = \ \{ \ {\text{\rm{(facet}}} \ F \subset \ {\ScptB (\Gk )}_{N} \ | \ F \supsetneq E \ \} \ .
\end{equation}

\noindent For $N=1$, and $\ScptB (\Gk)_{N} = \ScptB (\Gk)$ the building with its original simplicial structure, we shall also refer to the above set as the {\color{\myblue} \it{sphere}} of $E$ and use the notation $\mysph (E)$.  \ Let
\begin{equation}\label{facet-b-prelim-ms}
E \ \rightarrow E'
\end{equation}

\noindent be the map which takes a facet $E \subset {\ScptB (\Gk )}_{N}$ to its facet closure in $\ScptB (\Gk )$.  Obviously 
$$
\Gk_{E,0} \ = \ \Gk_{E',0} \quad , \quad \Gk_{E,0} \ = \ \Gk_{E',0} \ , \ \ \ {\text{\rm{so}}} \ \ \ \Gk_{E,0}/\Gk_{E,0^{+}} \, = \, \Gk_{E',0}/\Gk_{E',0^{+}} \ \ .
$$

\noindent Each facet $F \subset \myVsph (E)$ defines a parabolic subgroup $\Gk_{F,0}/\Gk_{E,0^{+}}$ of $\Gk_{E,0}/\Gk_{E,0^{+}}$, and this map is a surjection but for $N \ge 2$ not a bijection because the facet closure map from ${\ScptB (\Gk )}_{N}$ to  ${\ScptB (\Gk )}$ is not a bijection.  For $N=1$, we have

\begin{prop}\label{facets-prelim-ms}  The facets $F \subset \ScptB (\Gk)$ which properly contain a $E \subset \ScptB (\Gk )$ are in bijection with the spherical Tits building of $\Gk_{E,0}/\Gk_{E,0^{+}}$.
\end{prop}
\begin{proof}  The proper parahoric subgroups contained in $\Gk_{E,0}$ are in one-to-one correspondence with the proper parabolic subgroups of the finite field group $\Gk_{E,0}/\Gk_{E,0^{+}}$, and the later are parametrized by the facets of the spherical Tits building of   $\Gk_{E,0}/\Gk_{E,0^{+}}$.
\end{proof}

\smallskip

\bigskip

\subsection{Iwahori decompositions} \quad 

\medskip

Fix $N \in \bN_{+}$, and a facet $E \subset \ScptB (\Gk)_{N}$.  Suppose $\ScptA = \ScptA (\Sk )$ is an apartment containing $E$.   Let $r \in \frac{1}{N}\mathbb Z$. 
 The set 
\begin{equation}\label{constant-prelim-ms}
\Phi_{E} \ := \ \{ \ \mathrm{grad}(\psi) \ \ | \ {\text{\rm{$\psi$ a virtual  affine root,  $\psi=r$  on $E$}}} \ \} 
\end{equation}

\noindent is  a root sub-system of a $\Phi$, independent of $r$. For example, the extremes are:

\smallskip

\begin{itemize}
\item[(i)] \ If the facet $E$ is a special point then $\Phi_{E} = \Phi$. 

\smallskip

\item[(ii)] \ If $E$ is a chamber, then $\Phi_{E} = \emptyset$. 
\end{itemize}

\smallskip



\medskip

 For $\alpha \in \Phi \cup \{ 0 \}$, define:
\smallskip
$$
\aligned
\quad \psi_{\alpha , E , r}  \ :&= \ 
\begin{cases} 
\begin{array}{l}
{\text{\rm{the smallest virtual affine root (with gradient $\alpha$) so}}} \\
{\text{\rm{that $\psi_{\alpha , E , r}(x) \ge r$ \ $\forall \ x \in E$}}} \\
\end{array}
\end{cases} \\
&\ \\
\quad \psi_{\alpha , \myint (E) , r^{+}} :&= \ 
\begin{cases} 
\begin{array}{l}
{\text{\rm{the smallest virtual affine root (with gradient $\alpha$) so}}} \\
{\text{\rm{that $\psi_{\alpha , \myint (E) , r^{+}}( x ) > r$ \ $\forall \ x \in \myint (E)$}}}
\end{array}
\end{cases}
\endaligned
$$

\smallskip


\noindent Then, 
$$
\aligned
\Gk_{E,r} \ &= \ \Big( {\underset 
{\tiny{\begin{matrix} {\alpha \, \in \, (\Phi (\Sk ) \cup \{ 0 \} )} \\ \ \end{matrix}}}  \prod } \Xk_{\psi_{\alpha , E , r}} \ 
\Big) \ = \ \Big( {\underset 
{\tiny{\text{\rm{$\alpha \in (\Phi_{E} \cup \{ 0 \})$}}}}  \prod } \Xk_{\psi_{\alpha , E , r}} \ 
\Big) \quad \Big( \ 
{\underset {\tiny{\begin{matrix} {\text{\rm{$\alpha \notin \Phi_{E}$}}} \\
\ \end{matrix} }} \prod } \Xk_{\psi_{\alpha , E , r}} \ \Big) \\ 
\Gk_{E,r^{+}} \ &= \ \Big( {\underset 
{\tiny{\begin{matrix} {\alpha \, \in \, (\Phi (\Sk ) \cup \{ 0 \} )} \\ \ \end{matrix}}}  \prod } \ \Xk_{\psi_{\alpha , \myint (E) , r^{+}}} \ 
\Big) \ = \ \Big( {\underset 
{\tiny{\text{\rm{$\alpha \in (\Phi_{E} \cup \{ 0 \})$}}}}  \prod } \Xk_{\psi_{\alpha , \myint (E) , r^{+}}} \ 
\Big) \quad \Big( \ 
{\underset {\tiny{\begin{matrix} {\text{\rm{$\alpha \notin \Phi_{E}$}}} \\
\ \end{matrix} }} \prod } \ \Xk_{\psi_{\alpha , \myint (E) , r^{+}}} \ \Big) \\
\endaligned 
$$

\noindent We note that:

\smallskip
\begin{itemize}
 \item[(i)] \ If $\alpha\in \Phi_E\cup \{0\}$, then the value of $\psi_{\alpha , E , r}$ (resp.~$\psi_{\alpha , \myint (E) , r^{+}}$) on $E$ is $r$ (resp.~$> r$).

\smallskip
\item[(ii)] \ Otherwise  $\psi_{\alpha , E , r} =  \psi_{\alpha , \myint (E) , r^{+}}$.

\end{itemize}

\smallskip

\noindent Therefore,

\begin{equation}\label{p-prelim-ms}
\Gk_{E,r} / \Gk_{E,r^{+}} \ = \ {\underset {\alpha \, \in \, (\Phi_{E} \cup \{ 0 \} )} \prod }  \ \Xk_{ \psi_{\alpha , E , r}} {\, } / {\, }\Xk_{\psi_{\alpha , \myint (E) , r^{+}}} \ .
\end{equation}

\bigskip


\medskip

We recall that if $\ScptA$ is a (real orthogonal) affine space, and ${\mathcal E}$ is an affine subspace of $\ScptA$, then there is a {\color{\myblue}{\it{unique affine orthogonal transformation}}} $R_{\mathcal E}$ which reflects a point of $\ScptA$ across ${\mathcal E}$.

\medskip

Fix $N \in \bN$, and let $E \subset \ScptB(\Gk )_{N}$ be a (refined) facet.  Define two (refined) facets $F$ and $\bar F$ containing $E$ to be {\color{\myblue}{\it{opposite with respect to $E$}}} if there exists an apartment $\ScptA = \ScptA (\Sk )$ containing both $F$ and $\bar F$, so that in $\ScptA$:
\begin{itemize}
\item[(i)]  The affine subspaces $\myAff (F)$ and $\myAff (\bar F)$ are equal.
\smallskip
\item[(ii)] $R_{\myAff (E)}(F) \, = \, \bar F$.
\end{itemize}

\medskip

\noindent Since the parahoric subgroup $\Gk_{F,0}$ acts transitively on the apartments containing $F$, if $F$ and $F'$ are opposite with respect to $E$ in one apartment apartment, they will be opposite with respect to $E$ in any apartment containing them.

\medskip

Suppose $F$ and $\bar F$ are $E$-opposite, and $\ScptA = \ScptA (\Sk )$ contains $F$ and $\bar F$.  
 If $\psi$ is a virtual affine root such that $\psi=r$ on $F$ then $\psi=r$ on $\bar F$ and vice versa. Hence 
\begin{equation}\label{q-prelim-ms}
\aligned
 \Phi_{\bar F} \ = \ \Phi_{F} \ . \\
\endaligned
\end{equation}

\medskip

Let $\psi$ be a virtual affine root such that $\psi=r$ on $E$, hence its gradient is an element in $\Phi_E$.  
 We consider two possibilities for $\psi$:

\smallskip

\begin{itemize}
\item[(i)] \ $\psi$ is constant on $F$ (and on $\bar F$). Then $\psi=r$ on $F$ (and on $\bar F$), and the gradient of $\phi$ is in $\Phi_F=\Phi_{\bar F}$.

\smallskip

\item[(ii)] \ $\psi$ is nonconstant on $F$ (which is true if and only if it is nonconstant on $\bar F$).  Here we have the dichotomy that either $\psi > r$  on  $\myint (F)$
(and $\psi<r$ on $\myint (\bar F)$),  or the other way around. 
Set $\Phi_{E,F}$ to be the set of gradients of $\psi$ such that $\psi=r$ on $E$ and $\psi > r$  on  $\myint (F)$ roots. It is a choice of positive roots (determined by the facet $F$) in the root system $\Phi_{E}$.  We have 
$\Phi_{E,F} \, = \, -\Phi_{E,\bar F}$, and  $\Phi_{E}$ is the disjoint union of 
$\Phi_F=\Phi_{\bar F}$, $\Phi_{E,F}$ and $\Phi_{E,\bar F}$. 

\end{itemize}

\smallskip

\noindent As a consequence, we have:

\begin{prop}\label{iwahori-a-prelim-ms} \quad Under the hypothesis on $N$ and $r$ of Lemma \ref{bf-prelim-ms}, suppose $E$ is a (refined) facet of $\ScptB (\Gk )_{N}$, 
 $F$, and $\bar F$ are (refined) facets which are opposite with respect to $E$.  Let $\ScptA (\Sk )$ be an apartment containing $F$ and $\bar F$. 
 Then
{\rm{(Iwahori decomposition)}}
\begin{equation}\label{iwahori-b-prelim-ms}
\aligned
\Gk_{E,r}/\Gk_{E,r^{+}} \ &= \ \Gk_{F,r^{+}}/\Gk_{E,r^{+}} \ \oplus \ \Gk_{F,r}/\Gk_{F,r^{+}} \ \oplus \ \Gk_{\bar F,r^{+}}/\Gk_{E,r^{+}} \ .
\endaligned 
\end{equation}

\end{prop}

\medskip

We recall each of piece of the Iwahori decomposition \eqref{iwahori-b-prelim-ms} is a $\bF_{q} = {\mathfrak O}_{\mk}/\wp_{\mk}$ vector space.  Set

\begin{equation}\label{pontryagin-a-prelim-ms}
\aligned
V_{E} \ :&= \ \Gk_{E,r} / \Gk_{E,r^{+}} \quad {\text{\rm{(the notation suppresses the dependence on $r$)}}} \\
V_{F} \ :&= \ \Gk_{F,r} / \Gk_{F,r^{+}} \ = \ \Gk_{\bar F,r} / \Gk_{\bar F,r^{+}} \\
V_{E,F} \ :&= \ \Gk_{F,r^{+}} / \Gk_{E,r^{+}} \quad , \quad V_{E,\bar F} \ := \ \Gk_{\bar F,r^{+}} / \Gk_{E,r^{+}} \ \ .
\endaligned
\end{equation}

\bigskip
\bigskip

\medskip

    
\bigskip

\subsection{Convolution of {\,}$e_{\Gk_{E,r^{+}}}${\,} and {\,}$e_{\Gk_{F,r^{+}}}${\,} }   \quad \ 

\medskip
Henceforth we fix the natural number $N$. To simplify notation and terminology, affine roots, facets etc will mean their refined versions. 
Furthermore, if $\psi$ is an affine root and $F$ a facet $\psi >r$ on $E$ shall typically mean on the interior of $E$. 

\smallskip 
We fix a Haar measure on $\Gk$, and therefore a convolution structure on $\Cic (\Gk)$.   For any open compact subgroup $J \subset \Gk$ we define the idempotent:
\begin{equation}\label{haar-prelim-ms}
e_{J} \ := \ {\mfrac{1}{\meas ( J )}} \ 1_{J} \ .
\end{equation}

\medskip

\medskip

\begin{prop}\label{convolution-prelim-ms}  Suppose $E, F \subset \ScptB_{N}$. Let $C(E,F)$ be the combinatorial convex hull of $E$ and $F$. 
Let $D_E, D_F \subset C(E,F)$ be the unique facets so that {\ }(i) $E \subset D_{E}$, $F \subset D_{F}$, and {\ }ii) $\dim (D_{E}) = \dim (D_{F}) = \dim (C(E,F))$.  Then,
\begin{equation}\label{convolution-a-prelim-ms}
e_{\Gk_{E,r^{+}}} \ \star \ e_{\Gk_{F,r^{+}}} \ = \  e_{\Gk_{E,r^{+}}} \ \star \ 
e_{\Gk_{D_{E},r^{+}}} 
\ \star \ 
e_{\Gk_{D_{F},r^{+}}}
\ \star \ e_{\Gk_{F,r^{+}}} \ .
\end{equation}
\end{prop}


\medskip

\begin{proof} \ \ If $D_{E} = D_{F}$ (denote by $D$), then $E$ and $F$ are subfacets of $D$.  If $\psi$ is an affine root so that it values are greater than $r$ on either $E$, or $F$, then obviously, since $E$ and $F$ are subfacets of $D$, the values of $\psi$ on $D$ must also be greater than $r$.  The assertion \eqref{convolution-prelim-ms} follows.  \ So, we can and do assume $D_{E} \ne D_{F}$.  Set 
$$
\aligned
S \ :&= \ \{ \ \psi \ {\text{\rm{affine root}}} \ | \ \psi {\text{\rm{ has constant value on $D_{E}$ (and therefore on $D_{F}$ too)}}} \ \} \\
P \ :&= \ \{ \ \psi \ {\text{\rm{affine root}}} \ | \ \psi (D_{F}) > \psi (D_{E})  \ \} \\
N \ :&= \ \{ \ \psi \ {\text{\rm{affine root}}} \ | \ \psi (D_{F}) < \psi (D_{E})  \ \} \\
\endaligned
$$

\noindent Then, the convolution \eqref{convolution-a-prelim-ms} follows from the decompositions (in any order)
\begin{equation}\label{decomposition-a-prelim-ms}
\aligned
\Gk_{D_{E},r^{+}} \ &= \ \Big( {\underset {\tiny{\begin{matrix} {\psi \in N} \\ {\psi (D_{E}) > r} \end{matrix}}}  \prod } \ \Xk_{\psi} \ \Big) \ \Big( {\underset {\tiny{\begin{matrix} {\psi \in S} \\ {\psi (D_{E}) > r} \end{matrix}}} \prod } \ \Xk_{\psi} \ \Big) \ \Big( {\underset {\tiny{\begin{matrix} {\psi \in P} \\ {\psi (D_{E}) > r} \end{matrix}}}  \prod } \ \Xk_{\psi} \Big) \\
\Gk_{D_{F},r^{+}} \ &= \ \Big( {\underset {\tiny{\begin{matrix} {\psi \in N} \\ {\psi (D_{F}) > r} \end{matrix}}}  \prod } \ \Xk_{\psi} \ \Big) \ \Big( {\underset {\tiny{\begin{matrix} {\psi \in S} \\ {\psi (D_{F}) (=\psi (D_{E})) > r} \end{matrix}}} \prod } \ \Xk_{\psi} \ \Big) \ \Big( {\underset {\tiny{\begin{matrix} {\psi \in P} \\ {\psi (D_{F}) > r} \end{matrix}}}  \prod } \ \Xk_{\psi} \Big)  \ \ , \\
\endaligned
\end{equation}
\noindent and the inclusions
\begin{itemize}
\item[(i)]
\begin{equation}\label{decomposition-b-prelim-ms}
\Big( {\underset {\tiny{\begin{matrix} {\psi \in P} \\ {\psi (D_{E}) > r} \end{matrix}}}  \prod } \ \Xk_{\psi} \Big) \ \subset \ \Big( {\underset {\tiny{\begin{matrix} {\psi \in P} \\ {\psi (D_{F}) > r} \end{matrix}}}  \prod } \ \Xk_{\psi} \Big)
\quad{\text{\rm{and}}}\quad 
\Big( {\underset {\tiny{\begin{matrix} {\psi \in N} \\ {\psi (D_{F}) > r} \end{matrix}}}  \prod } \ \Xk_{\psi} \Big) \ \subset \ \Big( {\underset {\tiny{\begin{matrix} {\psi \in N} \\ {\psi (D_{E}) > r} \end{matrix}}}  \prod } \ \Xk_{\psi} \Big) \ .
\end{equation}

\item[(ii)]
\begin{equation}\label{decomposition-c-prelim-ms}
\Big( {\underset {\tiny{\begin{matrix} {\psi \in N} \\ {\psi (D_{E}) > r} \end{matrix}}}  \prod } \ \Xk_{\psi} \Big) \ \subset \ \Gk_{E,r^{+}} \quad{\text{\rm{and}}}\quad 
\Big( {\underset {\tiny{\begin{matrix} {\psi \in P} \\ {\psi (D_{E}) > r} \end{matrix}}}  \prod } \ \Xk_{\psi} \Big) \ \subset \ \Gk_{F,r^{+}}
\end{equation}
\end{itemize}
\noindent The inclusions of \eqref{decomposition-b-prelim-ms} are obvious.  To check the first inclusion of \eqref{decomposition-c-prelim-ms}, we need to show that $\psi (E) > r$.  If not, then $\psi (E) = r$.  Since $\psi \in N$, if follows that $r > \psi (D_{F})$; hence both $E$ and $F$ are contained in the half-space $\psi \le r$, while $D_{F}$ is not.  This contradicts $D_{F}$ being in the convex hull of $E$ and $F$.  Similarly for the second inclusion.
\end{proof}

\medskip

\noindent {\sc Remarks}: \quad (i) \  Since $E \subset D_{E}$, we have $\Gk_{D_{E},r^{+}} \supset \Gk_{E,r^{+}}$, and therefore $( e_{\Gk_{E,r^{+}}} \star  e_{\Gk_{D_{E},r^{+}}} ) = e_{\Gk_{D_{E},r^{+}}}$.  Similarly $( e_{\Gk_{D_{F},r^{+}}} \star  e_{\Gk_{F,r^{+}}} ) = e_{\Gk_{D_{F},r^{+}}}$.

\smallskip

\noindent (ii) \ Although not needed here, we note the proof of Proposition \eqref{convolution-prelim-ms} in fact shows that if $K$ is a subfacet of $C(E,F)$, then 
$$
e_{\Gk_{D_{E},r^{+}}} \ \star \ e_{\Gk_{D_{F},r^{+}}} \ = \ e_{\Gk_{D_{E},r^{+}}} \ \star \ e_{\Gk_{K},r^{+}} \ \star \  e_{\Gk_{D_{F},r^{+}}} \ .
$$

\vskip 0.70in 
 

\section{Representations of $\Gk_{E,r}/\Gk_{E,r^{+}}$}\label{pontryagin-ms}

\medskip

\subsection{Parabolic inflation and cuspidal characters} \ \quad

\smallskip

Let $V$ be a finite dimensional $\bF_{q}$ vector space, and  $V^{*} = \Hom_{\bF_{q}}(V ,  \bF_{q} )$ denote the dual space.  The choice of a non-trivial additive character $\chi$ of $\bF_{q}$ determines an identification of the Pontryagin dual  $\widehat{V}$ and the dual space $V^{*}$

\begin{equation}\label{dual-pontryagin-ms}
\aligned
V^{*} \ &\xrightarrow{\qquad } \ \widehat{V} \\
A \ &\xrightarrow{\qquad } \ \chi_{A} \, = \, \chi \circ A \ .
\endaligned
\end{equation}
\smallskip

Suppose $N \in \bN_{+}$, $r \in {\frac{1}{N}}\bN_{+}$, $E \subset \ScptB (\Gk )_N$ and $F, \, \bar F \, \in \myVsph (E)$ are opposite with respect to $E$.  
  The projection maps of the Iwahori decomposition

\begin{equation}\label{direct-pontryagin-ms}
\begin{matrix}
{\Gk_{E,r}/\Gk_{E,r^{+}}}  &=  &{\Gk_{F,r^{+}}/\Gk_{E,r^{+}}}  &\oplus  &{\Gk_{F,r}/\Gk_{F,r^{+}}}  &\oplus  &{\Gk_{\bar F,r^{+}}/\Gk_{E,r^{+}}} \\
\verteq  &\  &\verteq   &\   &\verteq  &\   &\verteq   &\  \\
V^{r}_{E}     &\  &V^{r}_{E,F}    &\   &V^{r}_{F}     &\   &V^{r}_{E,\bar F}   &\  \
\end{matrix}
\end{equation}

\medskip

\noindent to the subspaces $V^{r}_{E,F}$, $V^{r}_{F}$ and $V^{r}_{E,\bar F}$ give a natural identification
 of the dual spaces $(V^{r}_{E,F})^{*}$, $(V^{r}_{F})^{*}$, and $(V^{r}_{E,\bar F})^{*}$ as subspaces of $(V^{r}_{E})^{*}$ so that 
\begin{equation}\label{decomp-pontryagin-ms}
(V^{r}_{E})^{*} \ = \ (V^{r}_{E,F})^{*} \ \oplus \ (V^{r}_{F})^{*} \ \oplus \ (V^{r}_{E,\bar F})^{*} \ .
\end{equation}

\noindent Given $A \in (V^{r}_{E})^{*}$, we write its \eqref{decomp-pontryagin-ms} decomposition as:
\begin{equation}\label{decomp-2-pontryagin-ms}
A \ = \ A_{E,F} \ + \ A_{F} \ + \ A_{E,\bar F} \ .
\end{equation}

\smallskip


\begin{defn}\label{para-inflated-pontryagin-ms}
 Suppose $N \in \bN_{+}$, $r \in {\frac{1}{N}} \bN_{+}$, and $E \in \ScptB (\Gk )_{N}$.  Define $A \in (V^{r}_{E})^{*}$ to be {\color{\myblue} \it{parabolically inflated}} from a pair $(F,B)$ consisting of a facet $F \in \myVsph (E) \cup \{ E \}$ and $B \in  V^{*}_{F}$, if:

\smallskip 

\begin{itemize}
\item[(i)] \  The restriction of the character $\chi_{A} \in \reallywidehat{\Gk_{E,r} / \Gk_{E,r^{+}} }$ to $\Gk_{F,r{+}}/\Gk_{E,r^{+}}$ is trivial; therefore 
$\chi_{A}$ is the inflation of a character of $\Gk_{E,r}/\Gk_{F,r^{+}}$.
\smallskip
\item[(ii)]  \ The restriction of $\chi_{A}$ to $\Gk_{F,r}/\Gk_{F,r^{+}} (= V_{F}$) is $\chi_{B}$.  
\end{itemize}

\smallskip

\noindent If the facet $F$ belongs to $\myVsph (E)$, the above is equivalent to the existence of a facet $\bar F \in \myVsph (E)$, which is $E$-opposite to $F$, so that  the decomposition \eqref{decomp-2-pontryagin-ms} of $\chi_{A}$, has $A_{E,F} = 0$, and $A_{F} = B$, i.e., 
$$
 A \ = \ B + A_{E,\bar F} \ \ \ {\text{\rm{with \ $A_{E,\bar F} \in V^{*}_{E,\bar F}$}}} \  . 
$$
\end{defn}

\noindent We observe the property of being parabolically inflated is transitive, i.e., if $E \subset F \subset H$ are facets of $\ScptB (\Gk )_{N}$, and \ 
(i) $A \in V^{*}_{E}$ is parabolically inflated from $B \in V^{*}_{F}$ and \ (ii) $B$ is parabolically inflated from $C \in V^{*}_{H}$, then $A$ is parabolically inflated from $C \in V^{*}_{H}$.

\medskip

\begin{defn}\label{cuspidal-pontryagin-ms}
Define $A \in (V^{r}_{E})^{*}$ to be {\color{\myblue}{\it{non-cuspidal}}} if there exists facet $F \in \myVsph (E)$, i.e., $F \supsetneq E$, so that $A$ is parabolically induced from $V_{F}$, and {\color{\myblue}{\it{cuspidal}}} otherwise.  
\end{defn}

\noindent Equivalently, as a character on $\Gk_{E,r}$, $A$ is cuspidal if for any facet $F$ properly containing $E$, 
$$
{\text{\rm{the integral}}} \ \ \ \int_{\Gk_{F,r^{+}}} \ \chi_{A} (u) \, du \ \ \ {\text{\rm{vanishes}}} \ .  
$$

\vskip 0.30in

\noindent{\sc{Remarks and Examples}}.
\smallskip
\begin{itemize}
\item[(i)]  When $N=1$, $E \subset \ScptB (\Gk )$, and depth $r \in \bN_{+}$, the quotient group $\Gk_{E,r}/\Gk_{E,r^{+}}$ is isomorphic with the finite field Lie algebra $\Lie (\Gk_{E,0}/\Gk_{E,0^{+}} )$ of the group $\Gk_{E,0}/\Gk_{E,0^{+}}$,   The existence of a cuspidal character $A$ is equivalent to existence on an element in $\Lie (\Gk_{E,0}/\Gk_{E,0^{+}} )$ which does not lie in any proper parabolic subalgebra.

\medskip
 
\item[(ii)]  As an example, take the group $\SL (3)$

\begin{figure}[ht]
  \begin{center}
\includegraphics[height=5cm, angle=0]{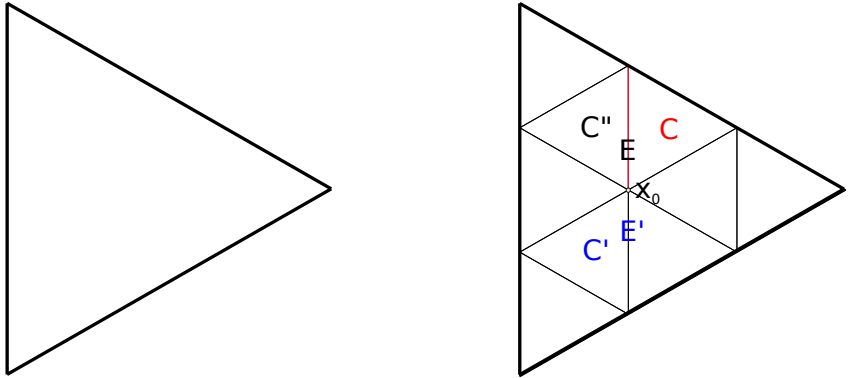} 
\caption{A2 chamber and refined A2 chamber ($N=3$)}
\label{figure-pontryagin-ms}
  \end{center}
\end{figure}

\item[(ii.1)] \ Take the point {\,}$x_0${\,} ({\,}$\alpha (x_0) = {\frac{1}{3}}$, $\beta (x_0) = {\frac{1}{3}}{\,})$, and depth $r = {\frac13}$.  \ \  Then:

$$
{\Gk_{x_{0},0}/\Gk_{x_{0},0^{+}}} = \left[ \begin{matrix} {\square} &{\quad } &{\quad } \\ {\quad } &{\square} &{\quad } \\ {\quad } &{\quad } &{\square} \end{matrix} \right] \quad , \quad 
{\Gk_{x_{0},{\frac{1}{3}}}/\Gk_{x_{0},{\frac{1}{3}}^{+}}} = \left[ \begin{matrix} {\quad } &{\square} &{\quad } \\ {\quad } &{\quad } &{\square} \\ {\square} &{\quad } &{\quad } \end{matrix} \right] \quad , \quad 
$$

\noindent with the boxed matrix entries indicating where the quotient is nontrivial.  \  For the $x_{0}$-opposite chambers ${\color{red}{C}}$, ${\color{blue}{C'}}$, the Iwahori decomposition of ${\Gk_{x_{0},{\frac{1}{3}}}/\Gk_{x_{0},{\frac{1}{3}}^{+}}}$ is: 
$$
{\Gk_{x_{0},{\frac{1}{3}}}/\Gk_{x_{0},{\frac{1}{3}}^{+}}} \ = \ {\color{red}{\left[ \begin{matrix} {\quad} &{\square} &{\quad} \\ {\quad} &{\quad} &{\square} \\{\quad} &{\quad} &{\quad} \end{matrix} \right]}} \, \oplus \, 
\{ 0 \} \, \oplus \, 
{\color{blue}{\left[ \begin{matrix} {\quad} &{\quad} &{\quad} \\ {\quad} &{\quad} &{\quad} \\{\square} &{\quad} &{\quad} \end{matrix} \right]}}
$$

\noindent For $x_{0}$-opposite edges ${\color{red}{E}}$, ${\color{blue}{E'}}$, the Iwahori decomposition is: 
$$
{\Gk_{x_{0},{\frac{1}{3}}}/\Gk_{x_{0},{\frac{1}{3}}^{+}}} \ = \ 
{\color{red}{\left[ \begin{matrix} {\quad} &{\quad } &{\quad} \\{\quad} &{\quad} &{\square} \\{\quad} &{\quad} &{\quad} \end{matrix} \right]}} \, \oplus \, 
\left[ \begin{matrix} {\quad} &{\square} &{\quad} \\ {\quad} &{\quad} &{\quad} \\{\quad} &{\quad} &{\quad} \end{matrix} \right] \, \oplus \, 
{\color{blue}{\left[ \begin{matrix} {\quad} &{\quad} &{\quad} \\ {\quad} &{\quad} &{\quad} \\ {\square} &{\quad} &{\quad} \end{matrix} \right]}}
$$

\item[(ii.2)] \ Take the facet {\,}$E${\,} ({\,}$\alpha (x_0) = {\frac{1}{3}}$, ${\frac{1}{3}} <\beta (x_0) < {\frac{2}{3}}{\,})$.  \ \  Then:

$$
{\Gk_{E,{\frac{1}{3}}}/\Gk_{E,{\frac{1}{3}}^{+}}} \ = \ \left[ \begin{matrix} {\quad } &{\square} &{\quad } \\ {\quad } &{\quad } &{\quad } \\ {\quad } &{\quad } &{\quad } \end{matrix} \right] \quad .
$$

\noindent For $E$-opposite chambers ${\color{red}{C}}$, ${\color{blue}{C''}}$, we have 
$$
{\Gk_{E,{\frac{1}{3}}}/\Gk_{E,{\frac{1}{3}}^{+}}} \ = \ 
{\color{red}{\left[ \begin{matrix} {\quad} &{\square} &{\quad} \\ {\quad} &{\quad} &{\quad} \\{\quad} &{\quad} &{\quad} \end{matrix} \right]}} \, \oplus \, 
\{ 0 \} \, \oplus \, 
{\color{blue}{\left[ \begin{matrix} {\quad} &{\quad} &{\quad} \\ {\quad} &{\quad} &{\quad} \\{\quad} &{\quad} &{\quad} \end{matrix} \right]}}
$$

\noindent  Here, every $A \in V^{*}_{E}$ is non-cuspidal.  
 
\smallskip

\end{itemize}

\medskip

\bigskip

\begin{prop}\label{property-pontryagin-ms} \quad Suppose $(E,B)$ is a cuspidal pair.

\smallskip

\begin{itemize}
\item[(i)] \ \ Take any apartment $\ScptA$ containing $E$.  

Suppose $\xi$ is a nonconstant affine function so that the hyperplane $H_{\xi}$ contains $E$; i.e., $\mygrad (\xi ) \, \perp \, E$. Then, there exists an affine root $\psi$ which have constant value $r$ on $E$, and 
\smallskip 
\begin{itemize}
\item[(i.1)]  \ $\langle \mygrad (\xi ) , \mygrad (\psi ) \rangle \, > \, 0$, 

\medskip

\item[(i.2)] \ $\Xk_{(\psi-r)}/\Xk_{(\psi-r)^{+}}$ is non-zero and $\chi_B$ restricted to $\Xk_{(\psi-r)}$ is non-trivial.
\end{itemize}

\medskip

\noindent Conversely, if $B \in (V^{r}_{E})^{*}$ satisfies the above, then $(E,B)$ is a cuspidal pair.

\medskip
\item[(ii)]  \ \ Under the inclusion $V^{*}_{E} \hookrightarrow V^{*}_{E} \otimes_{\bF_{q}} (\overline{\bF_{q}}) = \Hom_{\overline{\bF_{q}}}(\Gk^{\myun}_{E,r}/\Gk^{\myun}_{E,r^{+}} , \overline{\bF_{q}})$, the $\Gk^{\myun}_{E,0}/\Gk^{\myun}_{E,0^{+}}$-orbit of $B$ in $\Hom_{\overline{\bF_{q}}}(\Gk^{\myun}_{E,r}/\Gk^{\myun}_{E,r^{+}} , \overline{\bF_{q}})$ is closed.  
\end{itemize}
\end{prop}


\begin{proof} \quad  To prove assertion (i), in the given apartment $\ScptA$, set: 
$$
\aligned
F \ :&= \ \begin{cases}
\begin{tabular}{p{4.5in}}
the minimal facet containing $E$ so that for any interior point $y$ of $E$, the point {\,}$y \, + \, t \, \mygrad (\xi)  \in F${\,} for $t > 0$ sufficiently small,\\
\end{tabular}
\end{cases} \\
\bar F \ :&= \ \begin{cases}
\begin{tabular}{p{4.5in}}
the minimal facet containing $E$ so that for any interior point $y$ of $E$, the point {\,}$y \, + \, t \, \mygrad (\xi)  \in F${\,} for $t < 0$ sufficiently small. \\
\end{tabular}
\end{cases} \\
\endaligned
$$
The facets $F$ and $\bar F$ are $E$-opposite facets and provide an Iwahori decomposition of $V^{r}_{E}$ as in \eqref{decomp-pontryagin-ms}.  The cuspidality hypothesis on $(E,B)$ means $\chi_{B}$ is not trivial on $V^{r}_{E,F}$ (and $V^{r}_{E,{\overline{F}}}$ too).  The existence of the affine root $\psi$ satisfying (i.1) and (1.2) follows.  
\smallskip

To prove the converse, suppose $(E,B)$ satisfies the assumption of (i) (for any apartment $\ScptA$ that contains $E$).   If $F$ is a facet properly containing $E$, let $\xi$ be an affine functional which is constant on $E$, and ${\mygrad}(\xi )$ moves generically along $F$, i.e., for $x$ an interior point of $E$, then $x + t \mygrad (\xi )$ is in the interior of $F$ for small positive $t$.  Let $\psi$ be an affine root satisfying (i.1) and (i.2).  The hypothesis (i.1) that $\langle \mygrad (\xi ) , \mygrad (\psi ) \rangle > 0$ means $(\psi - r) > 0$ on $F$, so $\Xk_{(\psi - r)} \subset \Gk_{F,r^{+}}$.  That hypothesis (i.2) that $\chi_{B}$ restricted to $\Xk_{(\psi - r)}$ is nontrivial, means $\chi_{B}$ is not inflated from $\Gk_{E,r}/\Gk_{F,r^{+}}$.   As this is true for all $F \supsetneq E$, we deduce $(E,B)$ is a cuspidal pair.

\medskip

Assertion (ii) is a consequence of the 2nd fundamental result proved in [{\reKea}]. 
\end{proof}

\bigskip

The following noncuspidality Lemma also follows from the 2nd fundamental result proved in [{\reKea}]. 

\begin{lemma} Suppose $E \subset \ScptB_{N}$, and $B \in V^{*}_{E} \hookrightarrow V^{*}_{E} \otimes_{\bF_{q}} (\overline{\bF_{q}}) = \Hom_{\overline{\bF_{q}}}(\Gk^{\myun}_{E,r}/\Gk^{\myun}_{E,r^{+}} , \overline{\bF_{q}})$.  If there is a exists a one-parameter multiplicative subgroup $\lambda$ of $\Gk^{\myun}_{E,0}/\Gk^{\myun}_{E,0^{+}}$ which centralizes $B$, and acts nontrivially on $V^{*}_{E}$, then the pair $(F,B)$ is not cuspidal.
\end{lemma}

\begin{proof} \quad  The non-triviality of action of $\lambda$  on $V^{*}_{E}$ gives a nontrivial triangular decomposition $V^{*}_E \, = \, (V^{*}) \oplus (V^{*})^{0} \oplus (V^{*})^{-}$  where $(V^{*})^{0}$ is the fixed points of $\lambda$ and  $\lambda (t)$  acts as by positive powers of $t$ on $(V^{*})^{+}$ and negative on $(V^{*})^{-}$.  But, this decomposition arises from a pair of opposite chambers containing $E$, namely, take a point $x$ in the interior of $E$ and move along the line through $x$ slightly in the direction $\lambda$ (in both positive and negative direction) and this will give the interior points of $F$ and its $E$-opposite.
\end{proof}

\bigskip

Fix $N \in \bN_{+}$, and $r \in {\frac{1}{N}}\bN_{+}$.  For any pair $(E,A)$  of depth $r$, let $e_{E,A}$ denote the idempotent 
\begin{equation}\label{idem-char-pontryagin-ms}
e_{E,A} = \begin{cases} 
{\ } {\frac{1}{\meas (\Gk_{E,r} )}} \ \chi_{A} \quad &{\text{\rm{on $\Gk_{E,r}$}}} {\ }_{\begin{matrix} \ \ \end{matrix}}  \\
{\ }0 \quad &{\text{\rm{off $\Gk_{E,r}$}}}  {\ }^{\begin{matrix} \ \ \end{matrix}}
\end{cases}
\end{equation}
\noindent When the  character $\chi$ is given, without the corresponding $A$, we shall write $e_{E,\chi}$.

\begin{prop}\label{uniqueness-pontryagin-ms} \quad Given $A \in V^{*}_{E}$, 
there exists a cuspidal pair $(F,B)$ so that $A$ is parabolically inflated from the pair $(F,B)$.  In such a situation, 
$$
e_{E,A} \, \star \, e_{F,B} \ =  \ e_{E,A}. 
$$

\end{prop}

\smallskip

\begin{proof} \quad Use the transitivity of parabolic inflation, to deduce that there is a facet $F$ of maximal dimension so that $A$ is parabolically inflated from some $B \in V^{*}_{F}$.  That $B$ is cuspidal is because of the maximality assumption on $\dim (F)$.  The convolution formula is the orthogonality or characters $\widehat{V_{E}}$.

\end{proof}

\bigskip

\subsection{Some convolutions}\label{some-convolutions-pontryagin-ms}

We now come to the main result of this section. 

\medskip

\begin{prop}\label{key-prop-pontryagin-ms} Suppose $(E,A)$ and $(F,B)$ are two cuspidal pairs (of depth $r>0$).  If 
$e_{E,A} \ \star \ e_{F,B} \neq 0$ then $E$ and $F$ are aligned and $A=B$. 


\end{prop}

\begin{proof} \ \  Let  $D_{E}$ and $D_{F}$ be the unique facets (see Lemma \ref{facet-prelim-ms}) of $C(E,F)$ of dimension $\dim (C(E,F))$, and which contain $E$ and $F$ respectively.  Since $e_{E,A} \ \star \ e_{F,B} \neq 0$, it follows that 
$$
e_{E,A} \ \star \ e_{\Gk_{F,r^{+}}} \ \neq  \ 0 \ .
$$
Assume that $\dim (C(E, F)) > \dim (E )$.  We have 
$$
\aligned
e_{E,A} \ \star \ e_{\Gk_{F,r^{+}}} \ &= \ ( \, e_{E,A}  \ \star \ e_{\Gk_{E,r^{+}}} \, ) \ \star \ e_{\Gk_{F,r^{+}}} \ = \ e_{E,A}  \ \star \ ( \, e_{\Gk_{E,r^{+}}} \ \star \ e_{\Gk_{F,r^{+}}} \, )  \\
&= \ e_{E,A}  \ \star \ ( \, e_{\Gk_{D_{E},r^{+}}} \ \star \ e_{\Gk_{F,r^{+}}} \, ) \qquad {\text{\rm{(by Proposition \ref{convolution-prelim-ms})}}} \\
&= \ ( \, e_{E,A}  \ \star \ e_{\Gk_{D_{E},r^{+}}} \, ) \ \star \ e_{\Gk_{F,r^{+}}} \ \ .
\endaligned
$$
\noindent Since the pair $(E,A)$ is cuspidal and $E \subsetneq D_{E}$, we deduce the convolution $( \, e_{E,A}  \ \star \ e_{\Gk_{D_{E},r^{+}}} \, )$ is zero; so, 
$e_{E,A} \ \star \ e_{\Gk_{F,r^{+}}}$ is zero, a contradiction. Replacing the roles of $E$ and $F$, we conclude that 
$$ 
\dim (F) = \dim (C(E,F)) = \dim (E), 
$$
that is, $E$ and $F$ are aligned. 
Thus the cosets of $\Gk_{E,r^{+}}$ in $\Gk_{E,r}$ 
and  of $\Gk_{F,r^{+}}$ in $\Gk_{F_r}$ are represented by the same group elements, and it is easy to verify that 
$$ 
e_{E,A} \, \star \, e_{\Gk_F,r^{+}} \ = \  e_{E,A} \, \star \, e_{F,A}. 
$$
 Now
$$
\aligned
e_{E,A} \ \star \ e_{F,B} \ &= \  e_{E,A} 
\ \star \ \big( e_{\Gk_F,r^{+}} \ \star \ e_{F,B} \big) \\
&= \ \big( e_{E,A} \  \star \ e_{ \Gk_F,r^{+}}\big) \ \star \ e_{F,B} \\
&= \ \big( e_{E,A} \ \star \ e_{F,A} \big) \ \star \ e_{F,B} \\
&= \ e_{E,A} \ \star \ \big( e_{F,A} \ \star \ e_{F,B} \big)  \ \ . \\
\endaligned
$$
\noindent The convolution $e_{F,A} \ \star \ e_{F,B}$ is nonzero precisely when $A=B$.

\end{proof}

\bigskip

\begin{cor}   \label{key-cor-pontryagin-ms} 
Assume that  $(E_1,A_1)$ and $(F_1,B_1)$ are parabolically inflated from cuspidal pairs $(E,A)$ and $(F,B)$, respectively. If 
$e_{E_1, A_1} \ \star \  e_{F_1, B_1}\neq 0$  then $E$ and $F$ are aligned and $A=B$. In particular, this is true if
$(E_1,A_1)=(F_1,B_1)$. 

\end{cor}

\begin{proof}  By Proposition \ref{uniqueness-pontryagin-ms}  and associativity of convolution, 
$$ 
e_{E_1, A_1} \ \star \  e_{F_1, B_1}= \big( e_{E_1,A_1} \ \star  \ e_{E, A} \big)  \ \star \ \big( e_{F,B} \ \star  \ e_{F_1, B_1} \big)  = 
 e_{E_1,A_1} \ \star \ \big( e_{E, A}\ \star \ e_{F,B}  \big)\  \star  \ e_{F_1, B_1}. 
$$ 
Thus $e_{E, A}\ \star \ e_{F,B}\neq 0$. The corollary follows from Proposition \ref{key-prop-pontryagin-ms}.

\end{proof}

\vskip 0.50in

\subsection{Associate cuspidal characters} \  \quad

\medskip

Recall two facets $F_1, F_2 \, \subset \, \ScptB$ (the Bruhat-Tits building with its original simplicial structure) are {\color{\myblue}associate}, when there exists $g \in \Gk$ so that $g.F_1$ and $F_2$ are aligned, i.e., $\dim (F_1) = \dim (F_2) = \dim (C(g.F_1,F_2))$.   In this situation, for any $k \in \bN$, there is a canonical equality of the groups $\Gk_{g.F_{1},k}/\Gk_{g.F_{1},k^{+}}$ and $\Gk_{F_{2},k}/\Gk_{F_{2},k^{+}}$, and the element $g$ provides an isomorphism of the groups $\Gk_{F_{1},k}/\Gk_{F_{1},k^{+}}$ and $\Gk_{F_{2},k}/\Gk_{F_{2},k^{+}}$.  The analogue of this to two refined facets $F_1, \, F_2 \, \subset \, \ScptB_{N}$, and $r \in {\frac{1}{N}} \bN_{\ge 0}$ holds. 

\medskip

\begin{minipage}{\textwidth}
\begin{prop}\label{associate-facets-pontryagin-ms} \ \ 

\begin{itemize}
\item[(i)]  Suppose $F_1$, $F_2$ and $F_3$ are three facets with $F_1$ and $F_2$ associate to $F_3$.  Then, there exists $r, s \, \in \, \Gk$ so that $F_1$, $r.F_2$ and $s.F_3$ lie in an apartment $\ScptA$, and each of these facets generates the same affine subspace of $\ScptA$.

\smallskip

\item[(ii)] Associativity is an equivalence relation on the facets of $\ScptB_{N}$.  
\end{itemize}
\end{prop}
\end{minipage}

\medskip
\begin{proof} \ \ To prove  assertion (i), take $g, \, h \, \in \Gk$ so that $F_1$ and $g.F_2$ are aligned, and $F_2$ and $h.F_3$ are aligned (and thus $g.F_2$ and $(gh).F_3$ are aligned).   Let $\ScptA$ be an apartment containing $F_1$ and $g.F_2$, and let $\ScptA'$ be an apartment containing $g.F_2$ and $(gh).F_3$.  Take $k \in \Gk_{g.F_2,0}$ so that $k.\ScptA' = \ScptA$.  Then $F_1$, $g.F_2$ and $k.(gh).F_3$ are aligned in $\ScptA$.

\smallskip

Assertion (ii) obviously follows from assertion (i).

\end{proof}

\medskip

Fix $N \in \bN_{+}$ and $r \in {\frac{1}{N}}\bN_{+}$.  Define two cuspidal pairs $(E_1,A_1)$ and $(E_2,A_2)$ of depth $r$ to be {\color{\myblue}associate} if:
\smallskip

\begin{itemize}
\item[(i)] \ The facets $E_1$ and $E_2$ are associate, i.e., there exists $g \in \Gk$ so that $g.E_{1}$, $E_{2}$ are aligned -- so, $\Gk_{g.E_{1},r}/\Gk_{g.E_{1},r^{+}} = \Gk_{E_{2},r}/\Gk_{E_{2},r^{+}}$, and $Ad (g)$ provides an isomorphism between $\Gk_{g.E_{1},r}/\Gk_{g.E_{1},r}$ and $\Gk_{E_{1},r} / \Gk_{g.E_{1},r^{+}}$).

\smallskip

\item[(ii)] \ For an element $g$ as in part (i), there exists  $j \in \Gk_{E_{2},0}$ so that, as characters, $\myAd ( \, j \, ) E_2 = \myAd ( \, g \, ) (E_1)$.



\end{itemize}

\medskip

\noindent {\sc{Example.}} \quad Take $\Gk = \SL (2,\mk )$, and let $\Sk$ be the diagonal torus.  For $r \in \bN_{+}$, the groups $\Gk_{x,r}$ and $\Gk_{x,r^{+}}$ are constant in the interior of the chambers of $\ScptA (\Sk )$.  If $x$ an interior point of such a chamber, then the quotient group $\Gk_{x,r}/\Gk_{x,r^{+}}$ is supported on the diagonal elements and is isomorphic to $\bF_{q}$.   Fix a chamber $F \subset \ScptA (\Sk )$.  Then, with $A \in  \widehat{\bF_{q}}$ the cuspidal pairs $(F,A)$ and $(F,-A)$ are associate.  Note that for $p \ne 2$, and $A \ne 0$, the characters $A$ and $-A$ are not in the same $\Ad (\Gk_{x,0} )$-orbit in $\reallywidehat{\Gk_{x,r}/ \Gk_{x,r^{+}}}$.   


\bigskip

\begin{prop}\label{associattivity-cuspidal-pairs-pontryagin-ms} Associativity of cuspidal pairs is an equivalence relation. 
\end{prop}

\begin{proof} \ \ It suffices to prove transitivity.  Suppose $(F_2,A_2)$ is associate to $(F_1,A_1)$ and $(F_3,A_3)$.  We can assume:
\smallskip
\begin{itemize} 
\item[$\bullet$] \ $F_1$ and $g.F_2$ are aligned, and $k \in (\Gk_{F_1,0} \, \cap \, \Gk_{g.F_2,0}) $ satisfies $A_{1} = \Ad (k) (\Ad (g) A_{2} )$.
\smallskip  
\item[$\bullet$] \ $F_2$ and $h.F_3$ are aligned, and  $\Ad (g) A_{3} = \Ad (\ell) (\Ad (gh) A_{3} )$ with $\ell \in (\Gk_{g.F_2,0} \, \cap \, \Gk_{(gh).F_3} )$.
\end{itemize}

\smallskip

\noindent As in the proof of Proposition \eqref{associate-facets-pontryagin-ms} , take an apartment containing $g.F_2$ and $g.(h.F_3)$, and a $t$ translation of the apartment so that the facets $F_1$, $gF_2$ and $t(gh).F_3$ are simultaneously aligned.  Then, the above elements $k$ and $\ell$ can be assume to be in the intersection $(\Gk_{F_1,0} \, \cap \, \Gk_{g.F_2,0} \, \cap \, \Gk_{t.((gh).F_3)} )$, and we then deduce $(F_1,A_1)$ and $(F_3,A_3)$ are associate.  
 
\end{proof}

Set ${\mathcal C} = {\mathcal C} (E,A)$ to be the equivalence class of cuspidal pairs associate to a given cuspidal pair $(E,A)$.  The set ${\mathcal C}$ is trivially $\Gk$-equivariant, i.e., for any $(F,B) \in {\mathcal C}$, and $g \in \Gk$, then $(g F , \Ad (g) B)$ belongs to ${\mathcal C}$ too, and obviously ${\mathcal C}$ has an attached equivalence class of associate facets ${\mathcal F} ( {\mathcal C})$, namely the associativity class of $E$.  We call ${\mathcal F} ( {\mathcal C})$ the {\it{support}} of ${\mathcal C}$.  

\medskip

For $E' \in {\mathcal F} ( {\mathcal C})$, we define the idempotent:

\begin{equation}\label{cuspidal-idempotent-pontryagin-ms} 
e^{{\,}{\mathcal C}}_{E'} \ := \ {\underset {(E',B) \, \in \, {\mathcal C}} \sum} e_{E',B} \ .
\end{equation}      
\noindent This is a system of $\Gk$-equivariant idempotents supported on the facets ${\mathcal F} ( {\mathcal C})$.

\vskip 0.40in

Define the {\it{facet closure}} ${\overline{{\mathcal F}({\mathcal C})}}$ of ${\mathcal F}({\mathcal C})$ to be

\begin{equation}\label{closure-pontryagin-ms}
{\overline{{\mathcal F}({\mathcal C})}} \ := \ \{ \ {\text{\rm{$D$ a (refined) facet}}} \ | \ {\text{\rm{$D$ is a subfacet of a facet in ${\mathcal F}({\mathcal C})$}}} \ \} \ .
\end{equation}

\noindent   If $D \in  {\overline{{\mathcal F}({\mathcal C})}}$, set
\begin{equation}\label{local-star-pontryagin-ms}
\mystar_{\mathcal C} (D) \ := \  \{ \  {\text{\rm{facets $F \subset {{\mathcal F}({\mathcal C})}$}}}  \ | \ D \subset F \ \} \ \ , 
\end{equation}     
\noindent  and 

\begin{equation}\label{idempotent-pontryagin-ms}
\aligned
\chi (D) \ :&= \ \begin{cases}
\begin{tabular}{p{3.7in}}
set of characters of $\Gk_{D,r}/\Gk_{D,r^{+}}$ parabolically inflated from a cuspidal pair $(F,B) \in {\mathcal C}$, with $F \in \mystar_{\mathcal C} (D)$ \\
\end{tabular}
\end{cases} \\
e^{{\,}{\mathcal C}}_{D} \ :&= \ {\text{\rm{$ \,  \ {\underset {\chi \in \chi (D)} {\sum}} \ e_{D,\chi}  \, $ \ \ \ (an idempotent).}}} 
\begin{matrix} \ \\ \ \\ \ \end{matrix}
\endaligned
\end{equation}

\smallskip

\noindent Note that if the facet $D$ is a facet of ${\mathcal F} ({\mathcal C} )$, then the above definition agrees with $e^{{\,}{\mathcal C}}_{D}$ as defined in \eqref{cuspidal-idempotent-pontryagin-ms}.   

\smallskip 

For any $D \in {\overline{{\mathcal F}({\mathcal C})}}$ and  $F \in \mystar_{\mathcal C}(D)$, then trivially,
$$
e^{{\,}{\mathcal C}}_{F} \, \star \, e^{{\,}{\mathcal C}}_{D} \ = \ e^{{\,}{\mathcal C}}_{F}  \ = \ e^{{\,}{\mathcal C}}_{D} \, \star \, e^{{\,}{\mathcal C}}_{F} \ .
$$

\noindent More generally, we have:

\begin{prop} Suppose $D \, , \, E \, \in \, {\overline{{\mathcal F}({\mathcal C})}}$, with $D \subset E$. \ Then 
$$
e^{{\,}{\mathcal C}}_{E} \, \star \, e^{{\,}{\mathcal C}}_{D} \ = \ e^{{\,}{\mathcal C}}_{E}  \ = \ \, e^{{\,}{\mathcal C}}_{D} \, \star \, e^{{\,}{\mathcal C}}_{E} \ .
$$
\end{prop}

\begin{proof}  The assertion follows from the property that 
$D \subset E$ means 
$$
\mystar_{\mathcal C}(D) \ \supset \ \mystar_{\mathcal C}(E) \ . 
$$
\end{proof}

\noindent As a complement, we have:

\medskip

\begin{prop} Suppose $E \, , \, F \, \in \, {\overline{{\mathcal F}({\mathcal C})}}$, with $E \cap F \neq \emptyset$. \ Then 
$$
e^{{\,}{\mathcal C}}_{E} \, \star \, e^{{\,}{\mathcal C}}_{F} \ = \ \, e^{{\,}{\mathcal C}}_{F} \, \star \, e^{{\,}{\mathcal C}}_{E} \ .
$$
\end{prop}

\begin{proof} \quad The assertion follows from the fact that the idempotents {\,}$e^{\mathcal C}_{E}${\,} and {\,}$e^{\mathcal C}_{F}${\,} factor to functions on the abelian quotient group {\,}$\Gk_{(E \cap F),r}/\Gk_{(E \cap F),r^{+}}$.   
\end{proof}

\vskip 0.70in 
 

\section{Some Euler--Poincar{\'{e}} sums}\label{ep-ms}

\medskip

We fix $N \in \bN_{+}$, and $r \in {\frac{1}{N}}{\bN_{+}}$.  Suppose $C$ is a chamber in $\ScptB_{N}$, and $K$ is a facet of $C$.  Define
\begin{equation}\label{a-ep-ms}
\myVsph_{C} (K) \ := {\text{\rm{set of subfacets of $C$ which contain $K$}}} . 
\end{equation}


Let
\begin{equation}\label{min-para-ep-ms}
\aligned
\Delta_{C}(K) \, :&= \{ \, F \in \myVsph_{C}(K) \ | \ \dim (F) = (\dim (C) - 1) \ \} \ \subset \myVsph_{C}(K) 
\endaligned
\end{equation}
\noindent of facets of codimension one.  
we have a bijection between subsets of $\Delta_{C}(K)$ and $\myVsph_{C}(K)$ given by

$$
Q = \{ \, F_{j_1} , \, \dots \, , \, F_{j_s} \, \} \, \subset \, \Delta_{C}(K) \quad \longleftrightarrow \quad {\underset {{F_{j_{i}}} \in Q } \bigcap F_{j_{i}} } \ \in \ \myVsph_{C}(K) \ \ ,   
$$
\noindent with the convention that when $Q = \emptyset$, the empty intersection is the chamber $C$.

\bigskip



\begin{prop}\label{g-ep-ms} Suppose $N \in \bN_{+}$, $r \in {\frac{1}{N}}{\bN_{+}}$.  Suppose $C$ is a chamber of $\ScptB_{N}$, and $J \not\subset K$ are facets of $C$.  Then
\smallskip
\begin{itemize}
\item[(i)]  
$$
\big( \ {\underset {F \in \myVsph_{C}(K)} \sum } \ (-1)^{\dim (F)} \ e_{\Gk_{F,r^{+}}} \ \big) \ \star \ e_{\Gk_{J,r^{+}}} \ = \ 0 \ .
$$ 
\smallskip
\item[(ii)]  Suppose ${\mathcal C} = \{ (E,A) \}$ is an equivalence class of cuspidal pairs of depth $r$.
$$
\big( \ {\underset {F \in \myVsph_{C}(K)} \sum } \ (-1)^{\dim (F)} \ e^{\mathcal C}_{F} \ \big) \ \star \ e_{\Gk_{J,r^{+}}} \ = \ 0 \ .
$$ 
\end{itemize}

\end{prop}

\begin{proof}  \quad We observe that any two facets $E_1, \, E_2$ of $C$, the convolution $e_{\Gk_{E_1,r^{+}}} \star \, e_{\Gk_{E_2,r^{+}}}$ equals $e_{\Gk_{C(E_1,E_2),r^{+}}}$, where $C(E_1,E_2)$ is the convex closure of $E_1$ and $E_2$.  We deduce from this observation, that we can replace $J$ with the facet $C(J,K) \supsetneq K$ (equality is not possible since $J \not\subset K$).   \ So we assume $J \supsetneq K$.  This assumption means we have a map

\begin{equation}\label{j-ep-ms}
\aligned
\myVsph_{C}(K) \ \ &\longrightarrow \ \ \ \myVsph_{C}(J) \\
E \ \ \ &\longrightarrow \ \ C(E,J) \ .
\endaligned
\end{equation}

\noindent As above, let  $\Delta_{C}(K) := \, \{ \, F_1 \, , \, \dots \, , F_{k} \, \}$ be the facets of dimension $(\dim (C) - 1)$.  We assume the numbering of the facet is so that the intersection $(F_1 \cap \dots \cap F_j )$ equals $J$, i.e., $\{  \, F_1 \, , \, \dots \, F_j \, \}$ equals $\Delta_{C}(J)$.  \  Then, the map of \eqref{j-ep-ms} is 
\smallskip
$$
{\text{\rm{$Q \ \subset \ {\Delta}_{C}(K) \quad \longrightarrow \quad (Q \cap {\Delta}_{C}(J)) \ \subset \ {\Delta}_{C}(J)$ \ ,}}} 
$$
\noindent from which see immediately see that the fiber of a subset $Q' \subset 
{\Delta}_{J}'$ consists of the $2^{(k-j)}$ subsets 
$$
Q' \ \cup \ X {\text{\rm{\quad , \quad where $X$ is a subset of $\Delta_{C}(K) \ \backslash \ {\Delta}_{C}(J)$.}}}
$$
\noindent Since $k>j$ (the assumption $J \supsetneq K$), the convolution sum over each fiber of  \eqref{j-ep-ms} is obviously $0$, and assertion (i) follows.

\medskip

The proof of assertion (ii) is similar to that of assertion (i).  We use the fact that 
$$
e^{\mathcal C}_{F} \ \star \ e_{\Gk_{J,r^{+}}} \ = \ e^{\mathcal C}_{C(F,J)} \ .
$$
\end{proof}




\vskip 0.50in


\vskip 0.70in 
 

\section{A Key Proposition}\label{keyprop-ms}

\medskip 


We fix a maximal $\mk$-split torus $\Sk = \mS (\mk) \subset \Gk = \mG (\mk )$.   
Let $\Phi$ denote the set of roots of $\Sk$.  
Given a simple set of roots $\Delta \subset \Phi$, set:
$$
\aligned
\Phi^{+}_{\Delta} \ :&= \ {\text{\rm{set of positive for $\Delta$}}} \\
\ScptA_{\Delta} \ :&= \ {\text{\rm{positive Weyl chamber for $\Delta$}}}
\endaligned
$$
Let $\ScptB (\Gk)_N $ be the refined Bruhat--Tits building of $\Gk$ and $\ScptA = \ScptA (\Sk )$ the apartment of the maximal split torus $\Sk$.  
Let $r\in \frac{1}{N}\mathbb Z$ and $K$ a (refined) facet in $\ScptA$. Recall that $\Phi_K$ is the set of gradients of affine roots $\psi$ such that $\psi=r$ on $K$.  
We record the following well known lemma. 
\smallskip
\begin{lemma}\label{facet-simple-roots-keyprop-ms}   Set $\ell = \dim (\ScptA ) \ (= \myrank (\mG))$.  Let $C \subset \ScptA$ be a (refined) chamber and let ${\mathcal F}(C) = \{ F_0 , F_1 , \dots , F_{\ell} \}$ be the faces of $C$.   For any nonempty proper subset ${\mathcal F}_{K} \subsetneq {{\mathcal F}(C)}$, set: 
\begin{itemize}
\item[$\bullet$] \ \ $K \ := \ {\underset {F \subset {\mathcal F}_{K}} \bigcap F}$. 
\item[$\bullet$] \ \ For $F \in {{\mathcal F}}_{K}$, let $\alpha_{F}$ be the roots of $\Sk$ so that $\alpha_{F} \perp F$, and points outward from $C$.
\end{itemize} 

\noindent Then the set \,  $\Delta_{K} \, := \, \{ \, \alpha_{F} \, | \, F \in {\mathcal F}_{K} \, \}$ is a set of simple root for the root system $\Phi_K$. 
 
\end{lemma}

\medskip

\begin{proof} Elementary.  The chamber $C$ maps to a Weyl chamber of $\Gk_{K,0}/\Gk_{K,0^{+}}$.
\end{proof}

Fix a chamber $C_{0} \subset \ScptA$, and let $\myht_{C_{0}}$ be the Bruhat height/length function (with respect to $C_{0}$) on the chambers of $\ScptB$.  For a chamber $D \neq C_{0}$, we recall the faces ${\mathcal F}(D)$ of $D$ can be partitioned into two nonempty sets:
\begin{equation}\label{child-parent-keypro-ms}
\aligned
{\text{\rm{$c_{C_{0}}(D)$ (children of $D$)}}} \ :&= \ {\text{\rm{chambers adjacent to $D$ and of length $(\myht_{C_{0}} (D) + 1)$}}} \\ 
{\text{\rm{$p_{C_{0}}(D)$ (parents of $D$)}}} \ :&= \ {\text{\rm{chambers adjacent to $D$ and of length $(\myht_{C_{0}} (D) - 1)$}}} \ .
\endaligned
\end{equation}
\noindent For a chamber $D \subset \ScptA$ and $F$ a face of $D$, denote by $s_{F}(D)$, the chamber obtained by reflecting (in $\ScptA$) the chamber $D$ across the face $F$, and define: 
\begin{equation}\label{definition-a-keyprop-ms}
\alpha_{F} \ := \  \begin{cases}
\begin{tabular}{p{4.6in}}
to be the root which is perpendicular to $F$ and which points inwards to (resp.~outwards from) $D$ if $\myht_{C_0}(s_{F}(D)) = (\myht_{C_0}(D)-1)$ (resp.~$\myht_{C_0}(s_{F}(D)) = (\myht_{C_0}(D)+1)$).
\end{tabular}
\end{cases}
\end{equation}
\medskip
\noindent Note that the notation $\alpha_{F}$ suppressed the dependence on the base affine chamber $C_{0}$ and apartment $\ScptA$.

\smallskip

Let $\ScptA \subset \ScptB$ be as Lemma \ref{facet-simple-roots-keyprop-ms}.   To a set of simple roots $\Delta$ of $\Sk$, and a chamber $C_{0} \subset \ScptA$, there is an associated sector $S(C_{0},\Delta )$ in $\ScptA$ whose definition we recall as:

\begin{equation}\label{definition-b-keyprop-ms}
S(C_{0},\Delta ) \ := \  \begin{cases}
\begin{tabular}{p{4.6in}}
smallest union of affine chambers so that if $x \in  S(C_{0},\Delta )$, and $v$ is in the positive Weyl chamber of the simple roots $\Delta$, then $(x +v) \in S(C_{0},\Delta )$ too.
\end{tabular}
\end{cases}
\end{equation}

\medskip

\noindent The apartment $\ScptA$ is the union of the sectors  $S(C_{0},\Delta )$ as $\Delta$ runs over the possible sets of simple roots.  See Figure \ref{figure-keypro-ms} for the example of C2 (where the eight sectors are color coded).

\smallskip 

\noindent A chamber $D$ belongs to $S(C_{0},\Delta )$ if there is a sequence of chambers  $C_{0} \, , \, C_{1} \, , \dots \, C_{r} = D$ with $i = \myht_{C_{0}} (C_{i})$ so that $C_{i+1}$ and $C_{i}$ share a face $F$ so that $\alpha_{F}$ is positive for $\Delta$.    

\medskip

\begin{prop}\label{proposition-keyprop-ms} Let $\ScptA \subset \ScptB = \ScptB (\Gk)$ be as Lemma \ref{facet-simple-roots-keyprop-ms}.    For any chamber $D \neq C_{0}$ in $\ScptA$, there exists a set (possibly more than one) of simple roots $\Delta$ of $\Sk$ so that: 

\medskip

\begin{itemize}  

\item[(i)] \ For every face $F \subset D$, the root $\alpha_{F}$ belongs to $\Phi^{+}_{\Delta}$.

\item[(ii)] \ $D$ is contained in $(C_{0} + \ScptA_{\Delta}) \ \subset S(C_{0} , \Delta )$ ($\ScptA_{\Delta}$ is the positive Weyl chamber of $\Delta$).

\item[(iii)] \ For every $\alpha \in \Delta$, there exists a face $F \subset D$ so that:
\begin{itemize}
\item[$\bullet$] \ $\alpha_{F}$ is outwards,
\item[$\bullet$] \ the root $\alpha$ appears in the expression of $\alpha_{F}$ in terms of $\Delta$. 
\end{itemize}
\end{itemize}
\end{prop}

\medskip

\begin{proof} \quad   Pick a special vertex $x_{0}$ of $C_{0}$ to be the origin.   There exists a unique set of simple roots $\Pi \subset \Phi$ such that $D\subset \ScptA_{\Pi}$.  Suppose $F$ is a face of $D$:

\smallskip

\begin{itemize}
\item[$\bullet$] \ If $F$ is not contained in the boundary of $\ScptA_{\Pi}$, let $x$ be an interior point of $F$. It easily follows, from the definition of $\alpha_F$, that  $\alpha_F(x) >0$ hence $\alpha_F\in \Phi_{\Pi}^+$.  

\smallskip

\item[$\bullet$] \ If $F$ is contained in the boundary of $\ScptA_{\Pi}$, then (since the boundary walls of $\ScptA_{\pi}$ are contained in simple root hyperplanes), either $\alpha_F$ or $-\alpha_F$ is in $\Pi$, and the boundary hyperplane containing $F$ is $\alpha_F=0$.  \ If  $D$ and $C_0$ are on the different sides of the hyperplane $\alpha_F=0$, then $\alpha_F\in \Pi$, in particular, it is positive. Otherwise $-\alpha_F\in \Pi$. 

\end{itemize}

\smallskip

\noindent Let 
$$
\Theta  \ := \ \begin{cases}
\begin{tabular}{p{4.5in}}
the set of faces  $F$ of $D$ contained in the boundary of $\ScptA_{\Pi}$ such that $D$ and $C_0$ are on the same side of the hyperplane $\alpha_F=0$
. \\
\end{tabular}
\end{cases}
$$

\smallskip

\noindent We identify $\Theta$ with a set 
 of simple roots $\Theta \subset \Pi$, by the map $F \mapsto -\alpha_F$. 
   Summarizing, for every face $F$ of $D$, either $\alpha_F\in \Phi_{\Pi}^+$ or $-\alpha_F \in \Theta$, and 
 every element of $-\Theta$ occurs as $\alpha_F$. 
 
 \medskip 
 \noindent  Claim: \ if $\alpha_F\in \Phi_{\Pi}^+$ then it is not a linear combination of roots in $\Theta$. 

\smallskip

\noindent{\it{Proof of claim}}: \  Since $D\neq C_0$, $\Theta$ is a proper subset of $\Pi$. Since $D$ is a simplex no proper subset of $\alpha_F$ is linearly 
  dependent. The claim now follows since $|\Theta| < |\Pi|$ and the number of faces of $D$ is $|\Pi| + 1$. 
  
\smallskip 
  
 \noindent Summarizing again, the set of all $\alpha_F$ is a union of $-\Theta$ and a subset of $\Phi_{\Pi}^+ \setminus \Phi^+_{\Theta}$, where $\Phi_{\Theta}$ denotes the root subsystem spanned by $\Theta$. Let $W_{\Theta}$ be the Weyl group of this root system, and let $w_{\Theta}\in W_{\Theta}$ be the longest 
 element.  We observe that $w_{\Theta}$ has order $2$ so we can write $w_{\Theta}$ at places where naturally one should have $w_{\Theta}^{-1}$.   

\medskip

We are now ready to prove the three statements of the Proposition.  Define 
$$
\Delta \ := \ w_{\Theta}(\Pi) \ . 
$$
\noindent Since $ w_{\Theta} (\Theta)=-\Theta$ and $w_{\Theta}$ permutes $\Phi_{\Pi}^+ \setminus \Phi^+_{\Theta}$, it follows that 
 $$
 \Phi^+_{\Delta} = (\Phi_{\Pi}^+ \setminus \Phi^+_{\Theta}) \cup \Phi^+_{-\Theta} 
 $$ 
\noindent hence all $\alpha_F\in \Phi^+_{\Delta}$, so statement (i) holds.

 \medskip 
 We now show statement (ii), $D$ is contained in $S(C_{0},\Delta)$. To this end, we consider the facet $E$ of $D$ defined by 
 \[ 
 E = {\underset {F \in \Theta} {\bigcap}} \ F \ . 
 \] 
 \noindent and observe that $w_{\Theta}$ fixes $E$. 
 Let $y$ be an interior point in $E$. Since $w_{\Theta} (y) =y$, it follows that $y\in \ScptA_{\Delta}$.  Since $C_0$ is on the same side as $D$ with respect to 
 hyperplanes spanned by faces $F\in \Theta$, it follows that $x+y$ is in the interior of $D$ for all $x$ in the interior of $C_0$, sufficiently close to $x_0$, so (ii) holds. 
 
\smallskip 
It remains to show statement (iii), for every $\alpha\in \Delta$ there exists an outward $\alpha_F$ such that $\alpha$ appears in the expression of $\alpha_F$ in 
terms of $\Delta$. If $\alpha \in -\Theta$, then $\alpha=\alpha_F$, for some face $F$, and $\alpha_F$ is outward by the definition of $\Theta$.  
Now assume that $\alpha \notin -\Theta$. 
Let $v\in D$ be the (unique) point maximizing the distance from $x_0$. Let $E$ be the facet of $D$ containing $v$ as the interior point. We identify $v$ with 
(half of) the gradient of the distance function at $v$. Then $v$ is perpendicular 
to $E$, and $v$ can be uniquely written down as a linear combination of vectors perpendicular to faces $F$ containing $E$ and pointing out of $D$.  Thus, 
\[ 
v= \sum_{F\supseteq E} x_F \alpha_F 
\] 
where $x_F>0$ and $x_F <0$ if $\alpha_F$ is outward and inward, respectively. 
Let $\lambda$ be the fundamental co-weight corresponding to $\alpha$, that is,  $\lambda(\alpha)=1$  and $\lambda(\beta) =0$ for all other $\beta\in \Delta$. 
If $\lambda(v) >0$ then 
\[ 
\lambda(v)= \sum_{F\supseteq E} x_F\lambda(\alpha_F) >0 
\] 
so there must be an outward $\alpha_F$ containing $\alpha$ in its linear expansion in terms of $\Delta$.  It remains to prove that $\lambda(v) >0$. 
Note that $v\in \ScptA_{\Pi}$, since $D\subset \ScptA_{\Pi}$. Observe that $\ScptA_{\Pi} \setminus 0$ is 
contained in the interior of the cone dual to $\ScptA_{\Pi}$.  Since $\ScptA_{\Pi}$ is spanned by fundamental co-weights,  one of which is $w_{\Theta}(\lambda)$, 
 it follows that $w_{\Theta}(\lambda)(v) >0$. Hence 
 \[ 
0< w_{\Theta}(\lambda)(v)=  \lambda(w_{\Theta}(v))=\lambda(v) 
\] 
 where the second equality follows from the fact that $v-w_{\Theta}(v)$ is a linear combination of roots in $\Theta$, so perpendicular to $\lambda$, 
 since $\alpha \notin -\Theta$. 

\end{proof}

\medskip

\begin{figure}[ht]
\includegraphics[height=15cm]{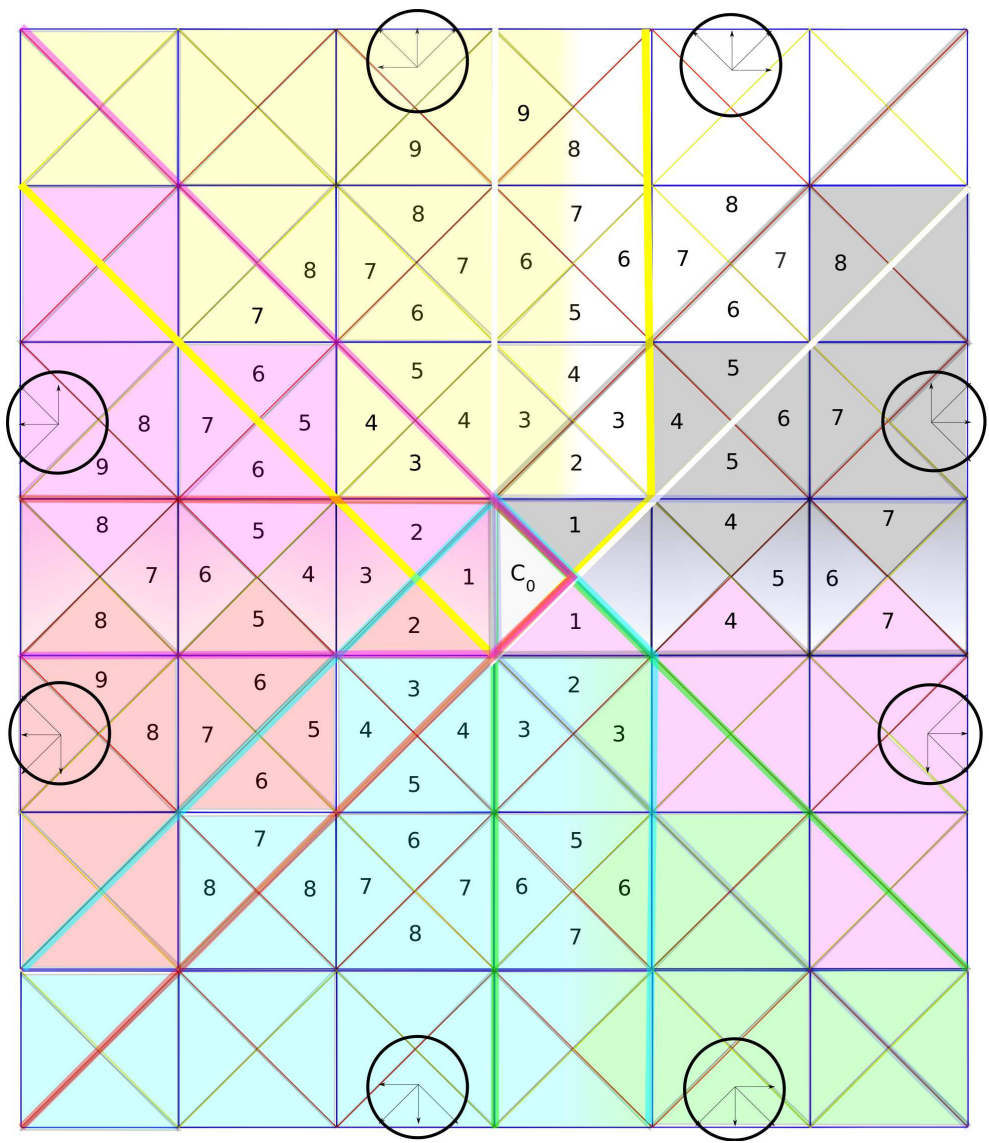}
\caption{Illustration of Proposition \ref{proposition-keyprop-ms} }
\label{figure-keypro-ms}
\end{figure}

\noindent {\sc{Example}}. \quad Figure \ref{figure-keypro-ms} illustrates Proposition \ref{proposition-keyprop-ms} for C2.  For the shown chamber $C_{0}$ the eight sectors $S(C, \Delta )$ are color coded with an indication of the positive roots in the circles.  Chambers of a single color have a unique set of simple roots $\Delta$ which satisfy the conditions of Proposition \ref{proposition-keyprop-ms}.  Chambers with a blend of two colors have two possible adjacent set of simple roots satisfying Proposition \ref{proposition-keyprop-ms}.
 
\vskip 0.70in 
 

\section{Essential Compactness of Euler-Poincar{\'{e}} Sums}\label{convolve-ms}

\medskip

\subsection{\label{a-convolve-ms}{Assumptions}} \

\medskip 

\noindent In the following, we assume: 

\begin{itemize}
\item[$\bullet$] \ $\mk$ is a non-archimedean local field.
\smallskip
\item[$\bullet$] \ $\Gk = \mG (\mk )$ is the group of $\mk$-rational points of an absolutely quasisimple split linear algebraic group defined over $\mk$.
\smallskip
\item[$\bullet$] \ $\ScptB = \ScptB (\Gk )$ is the Bruhat--Tits building of $\Gk$.  The assumption $\mG$ is quasisimple means the chambers of $\ScptB$ are simplicies. 
\smallskip
\item[$\bullet$] \ For $N \in \bN$, let $\ScptB_{N}$ denote the refined building, and suppose $r \in {\frac{1}{N}}\bN_{+}$.
\end{itemize}

\bigskip

\subsection{\label{b-convolve-ms}{Bruhat length/height on a refined building}} \ 

\medskip

The fixture of a chamber $C_0 \subset \ScptB_{N}$, leads to an elementary and obvious generalization of the Bruhat length/height function on the chambers of $\ScptB$ to a (generalized) Bruhat length/height function $\myht_{C_{0}}$ on the chambers of $\ScptB_{N}$.  As in  [{\reBCM}:{\S}2], if $\ScptA = \ScptA (\Sk )$ is an apartment containing $C_{0}$, and $\pm \alpha \in \Phi (\Sk )$ is a pair of opposite roots, we define  
\begin{equation}
\myht^{\pm \alpha}_{C_{0}} (D) \ := \ \begin{cases} 
\begin{tabular}{p{3.3in}}
the number of refined affine hyperplanes $H_{\psi}$ perpendicular to $\pm \alpha$ and separating $C_{0}$ and $D$.
\end{tabular} 
\end{cases}
\end{equation}

\noindent Then, $\myht_{C_{0}}(D)$ is the sum of the $\myht^{\pm \alpha}_{C_{0}} (D)$'s over all pairs of opposite roots of $\Phi (\Sk )$.  For $m \in \bN$, define, as in [{\reBCM}:{\S}2], the ball of radius $m$:
\begin{equation}\label{dist-ball-convolve-ms}
\myBall (C_{0},m) \ := \ \{ \ {\text{\rm{$D$ chamber of $\ScptB_{N}$}}} \ | \ \myht_{C_{0}}(D) \, \le m \, \} \ .
\end{equation}

\noindent Suppose $D$ is a chamber with $\myht_{C_0}(D) = m$.  We recall (see [{\reBCM}:{\S}2]):
\begin{itemize}
\item[$\bullet$] \ A face {\,}$F${\,} of {\,}$D${\,} is called {\it{inward}} with respect to the base chamber $C_{0}$ if {\,}$F${\,} is a subfacet of {\,}$\myBall (C_{0},(m-1))$.   A face {\,}$F${\,} of {\,}$D${\,} which is not inward is called {\it{outward}}.

\smallskip

\item[$\bullet$] \ Set 
\begin{equation}
\aligned
\mychild_{C_{0}}(D) \ :&= \ {\text{\rm{set of outward (child) faces of $D$}}} \ , \\
\myparent_{C_{0}}(D) \ :&= \ {\text{\rm{set of inward (parent) faces of $D$}}}.
\endaligned
\end{equation}

\vskip 0.05in

\noindent and define the facet {\,}${\mathcal E}_{C_{0}}(D)$ as:

\begin{equation}\label{key-facet-convolve-ms}
{\mathcal E}_{C_{0}}(D) \ := \ {\underset {F \in \mychild_{C_{0}}(D)} \bigcap} F \ . 
\end{equation}

\noindent Then (see [{\reBCM}:{\S}2.2]), the subsets of $\mychild_{C_{0}}(D)$ parametrize the facets of $D$ not in {\,}$\myBall (C_{0},(m-1))$ as follows:
$$
\myTau \, \subset \, \mychild_{C_{0}}(D) \quad \xleftrightarrow{\qquad} \quad {\myFacet (T)} \, := \, {\underset {F \in \myTau} \bigcap } \ F \ . 
$$
\noindent The empty intersection is the chamber $D$ itself.  These facets of $D$ are also precisely the facets of $D$ containing ${\mathcal E}_{C_{0}}(D)$. 
\end{itemize}
 
\bigskip

\subsection{Convolution}

\begin{lemma}\label{lemma-a-convolve-ms} \ Under the assumptions of subsection {\eqref{a-convolve-ms}}, fix a (base) chamber $C_{0} \subset \ScptB_{N}$, and an open compact subgroup $J \subset \Gk$.  Suppose ${\mathcal C} = \{ (E,A) \}$ is an associate class of cuspidal pairs, and $e^{{\,}\mathcal C}_{F}$ is the corresponding $\Gk$-equivariant system of idempotents defined in \eqref{idempotent-pontryagin-ms}.  \ \ Then, there exists $L_{0} \in \bN_{+}$ (dependent on $C_0$ and $J$) so that  for any chamber {\,}$D$, with  $\myht_{C_0}(D) \ge L_{0}$, the Euler-Poincar{\'{e}} convolution
$$
e_{J} \ \star \ \big( \, {\underset {{\mathcal E}_{C_{0}}(D) \subset K \subset D} \sum} \ (-1)^{\dim ( K )} \, e^{{\,}\mathcal C}_{K } \ \big) \quad {\text{\rm{vanishes.}}}
$$
\end{lemma}

\begin{proof} \quad Let $r \in {\frac{1}{N}}\bN_{+}$ be the depth of the associate class ${\mathcal C}$.  As preliminary reduction, we choose $s \in {\frac{1}{N}}\bN_{+}$ so that both $s \ge r$ and $\Gk_{C_{0},s} \subset J$.  If suffices to show the assertion with $J$ replaced by $\Gk_{C_{0},s}$.  \


\medskip

We remark that because the Iwahori subgroup $\Gk_{C_{0},o}$ acts transitively on the apartments containing $C_{0}$ it is sufficient to fix an $\ScptA  = \ScptA (\Sk)$ containing $C_{0}$, and establish the assertion of the Lemma for chambers in $\ScptA$ sufficiently far from $C_{0}$.  So, we fix $\ScptA = \ScptA (\Sk ) \supset C_{0}$.

\medskip

Suppose $D \ne C_{0}$ is a chamber of $\ScptA$.  \ By Proposition \ref{proposition-keyprop-ms}, there is a simple set of roots $\Delta \subset \Phi (\Sk )$ so that 
$$
\aligned
D \ \subset \ C_{0} \ + \ \ScptA_{\Delta} \quad {\text{\rm{($\ScptA_{\Delta}$ the positive Weyl chamber $\ScptA_{\Delta}$ of $\Delta$),}}}
\endaligned
$$

\noindent and 
\begin{itemize}
\item[(i)] For each face $F \subset D$, the root $\alpha_{F}$ (notation of  Proposition \ref{proposition-keyprop-ms}) belongs to the positive roots $\Phi^{+}_{\Delta}(\Sk )$.
\smallskip
\item[(ii)] For any simple root $\alpha \in \Delta$, there exists an outwards face $F \subset D$ (i.e., $\alpha_{F}$ is outwards from $D$), so that the root $\alpha$ appears in the expression of $\alpha_{F}$ in terms of $\Delta$.
\end{itemize}

\medskip

For any choice of a set of simple roots $\Delta \subset \Phi (\Sk )$, we show the vanishing assertion of the Lemma is true for a $D \subset C_{0} + \ScptA_{\Delta}$ satisfying (i), (ii) and $\myht_{C_{0}}(D)$ sufficiently large.

\medskip

For any chamber $D (\neq C_{0}) \subset \ScptA$, define, as in \eqref{key-facet-convolve-ms}:
$$
E \ := \ {\mathcal E}_{C_{0}} (D) = F_0 \cap F_1 \cap \ldots \cap F_l \ ,
$$ 
the intersection of all outward faces $F_i$  of $D$. Let $\Phi_E$ be the set of gradients of affine roots $\psi$ such that $\psi=r$ on $E$.  
The set $\Phi_E$ is a root system, is independent of $r\in \frac{1}{N}\mathbb Z$,  and  
$$
\{\alpha_{F_0}, \alpha_{F_1}, \ldots ,\alpha_{F_l}\} 
$$ 
\noindent is a set of simple roots for $\Phi_E$. Thus, if $\psi$ is an affine root constant on $E$, then the gradient of $\psi$ is a sum $\sum_i x_i\alpha_{F_i}$ where all 
$x_i$ are either non-negative or non-positive integers. Let $F\supset E$ be the intersection of all outward faces with $F_0$ removed. (If $F_0$ is the only outward face then we set 
$F=D$.)  

\medskip

We make two remarks:

\smallskip

\begin{itemize}

\item[$\bullet$] If $\psi$ is an affine root such that $\psi=r$ on $E$ and $\psi>r$ on $F$ then the gradient of $\psi$ is $\sum_i x_i\alpha_{F_i}$ where $x_0<0$ and 
$x_i\leq 0$ for all $i$. To see this, let $\delta$ be a translation of $\ScptA$ such that $\alpha_{F_0}(\delta)=1$ and $\alpha_{F_i}(\delta)=0$ for $i\neq 0$. Let $x$ be an 
interior point of $E$.  Since $\alpha_{F_0}$ is outward, the ray starting at $x$ in the direction of $-\delta$ passes through the interior of $F$. Since $\psi>r$ on $F$, 
it follows that $x_0<0$, and $x_i\leq 0$ 
for all $i$, since all coefficients have the same sign.  
\smallskip

\item[$\bullet$] 
 Given $a\geq 0$, there exists $L_a \geq 0$ such that: if $D$ is a chamber in $C_0 + \ScptA_{\Delta}$ and $\mathrm{ht}_{C_0}(D) \geq L_a$ then there exists 
 a simple root $\beta\in \Delta$ such that $\beta(x-y) > a$ for all $y\in C_0$ and $x\in D$.  To see this, 
 consider the distance on  $\ScptA$ defined by 
 \[ 
d(x,y)=  \max_{\alpha\in \Delta} |\alpha (x-y)|, 
 \] 
 where  $x,y\in \ScptA$. 
 Now observe that  there exists $\epsilon >0$ such that $\alpha (x-y)>-\epsilon$ for all $x\in C_0$, $y\in C_0 + \ScptA_{\Delta}$ and all $\alpha\in \Delta$. 
 Let $b=\max(a, \epsilon)$.  Observe that a bounded set in $\ScptA$ contains only finitely many chambers.
  Thus if $L_a$ is large enough then for all 
 $D$ such that $\mathrm{ht}_{C_0}(D) \geq L_a$  we have  $d(x,y) >b$ for all $x\in D$ and all $y\in C_0$.  If $D$ is contained in $C_0 + \ScptA_{\Delta}$ this simply means that 
 there exists $\beta \in \Delta$ such that $\beta (x-y)>b \geq a$. 

\end{itemize}

\medskip

We continue the proof of the Lemma under our stated assumptions on the chamber $D$ (that $D \subset C_{0} + \ScptA_{\Delta}$ satisfies hypotheses (i), (ii)). 
 Let $a = (s-r) + h\epsilon$,  where $h$ is the height of the highest root, and $\epsilon$ as in the second bullet above. Then, by the same bullet,  there exists $L_{a}$ be such that if 
 $\myht_{C_{0}}(D) \ge L_{a}$, then there exists $\beta \in \Delta$ so that $\beta (x) - \beta (y) \ge a$ for all $y \in C_{0}$ and $x \in D$.

\medskip

By hypothesis (ii), there exists an outward face $F_0$ of $D$ so that when $\alpha_{F_0}$ is written as an (integer) sum of simple roots (in $\Delta$), the coefficient of $\beta $ is $\ge 1$.  Let $F$ be the intersection of the outward faces of $D$ different from $F_0$.  
(So $F\supset E := {\mathcal E}_{C_{0}}(D)$.) We claim that
$$
e_{\Gk_{C_{0},s}} \ \star \ e_{\Gk_{E,r^+}} \ = \ e_{\Gk_{C_{0},s}} \ \star \ e_{\Gk_{F,r^{+}}} \ .
$$
 Let $\psi$ be a virtual affine root $\psi>r$ on $F$ but $\psi=r$ on $E$. To prove the identity it suffices to show that 
 $\mathfrak X_{\psi} \subset \Gk_{C_0,s}$ for every such $\psi$. By the first bullet above, 
 the gradient of $\psi$ is a sum  $\sum_i x_i \alpha_{F_i}$ where  
 $x_i\leq 0$ and $x_0<0$.   Since all $\alpha_{F_i}$ are positive roots, the gradient of $\psi$ is a negative root 
 containing a non-zero (integer) multiple of $\beta$ in its expression as a sum of roots in $\Delta$. Recall that if $x\in F$ and $y\in C_0$, 
 then $\beta(x-y) > a$ and $\alpha(x-y) > -\epsilon$ for all other $\alpha\in \Delta$. It follows that 
 $\psi(y-x) > s-r$. Thus 
 \[ 
 \psi(y)=\psi(x) + \psi (y-x) >s, 
 \]

\noindent  hence $\mathfrak X_{\psi} \subset \Gk_{C_0,s}$ which proves the identity. We apply it (and associativity of convolution) to obtain

$$
\aligned
e_{\Gk_{C_{0},s}} \ \star \ \big( &{\underset {E \subset K \subset D} \sum} \ (-1)^{\dim ( K )} \, e^{{\,}\mathcal C}_{K } \, \big) \ \ = \ \ e_{\Gk_{C_{0},s}} \ \star \ \big( e_{\Gk_{E,r^+}} \ \star \ \big( {\underset {E \subset K \subset D} \sum} \ (-1)^{\dim ( K )} \, e^{{\,}\mathcal C}_{K } \, \big) \, \big) \\
&= \ \ e_{\Gk_{C_{0},s}} \ \star \ \big( \ e_{\Gk_{F,r^{+}}} \ \star \ \big( {\underset {E \subset K \subset D} \sum} \ (-1)^{\dim ( K )} \, e^{{\,}\mathcal C}_{K } \, \big) \ \big) \, \ .
\endaligned
$$

\noindent By Proposition \ref{g-ep-ms} the inner convolution of the last line vanishes.  

\end{proof}

\medskip

\begin{cor}\label{stabilization-convolve-ms}  Suppose $f \in \Cic (\Gk)$. Under the assumptions of Lemma {\ref{lemma-a-convolve-ms}}, the convolution sums 
$$
f \ \star \ \big( {\underset {K \, \subset \, \myBall (C_{0},m) } \sum} \ (-1)^{\dim ( F )} \, e_{K} \ \big) \
$$
\noindent stabilize, i.e., are constant as a function of $m$, for $m$ sufficiently large.  
\end{cor}

\begin{proof}  \ The difference in the Euler-Poincar{\'{e}} convolution over the ball {\,}$\myBall (C_0,(m+1))$ and the ball $\myBall (C_0,m)${\,} is the sum of over the chambers $D \, \subset \, (\myBall (C_0,(m+1)) \backslash \myBall (C_0,m))$  of the Euler-Poincar{\'{e}} convolutions:
$$
e_{\Gk_{C_{0},s}} \ \star \ \big( {\underset {{\mathcal E}_{C_{0}}(D) \subset K \subset D} \sum} \ (-1)^{\dim ( K )} \, e^{{\,}\mathcal C}_{K } \ \big) \  ,
$$
\noindent which vanish by the Lemma \ref{lemma-a-convolve-ms}. 

\end{proof}

\bigskip

\begin{prop} \label{independence-a-convolve-ms} \ Under the assumptions of Corollary \eqref{stabilization-convolve-ms}, suppose $C_{1}$ is a chamber of $\ScptB_{N}$ which is adjacent to $C_{1}$, i.e., $C_{0}$ and $C_{1}$ share a common face.   Then, for m sufficiently large:
$$
f \ \star \ \big( {\underset {K \, \subset \, \myBall (C_{0},m) } \sum} \ (-1)^{\dim ( K )} \, e^{\mathcal C}_{K} \big) \ = \ f \ \star \ \big( {\underset {K \, \subset \, \myBall (C_{1},m) } \sum} \ (-1)^{\dim ( K )} \, e^{\mathcal C}_{K} \big) \ .
$$
\end{prop}

\begin{proof} \quad We first observe the fixture of a base chamber $C_{0}$ results in the following: \  Suppose $F \subset \ScptB_{N}$ is a face.  The convex closure $C(C_{0},F)$ contains a unique chamber $D$ which contains $F$.  Furthermore, if $D' \ne D$ is another chamber containing $F$, then $\myht_{C_{0}}(D') = \myht_{C_{0}}(D)+1$.  So, for any face $F \subset \ScptB_{N}$, we define:

\begin{equation}\label{unique-convolve-ms}
{\ScptC}_{C_{0}}(F) \ := \ {\text{\rm{the chamber $D$ containing $F$, so that $F$ is an outward face of $D$}}} . 
\end{equation}

\noindent The chamber ${\ScptC}_{C_{0}}(F)$ is the unique chamber of $C(C_{0},F)$   containing $F$ (denoted as $D_F$ in \eqref{facet-prelim-ms}).

\medskip

As a second observation, we note that the adjacency of $C_{0}$ and $C_{1}$ means $|(\myht_{C_0} - \myht_{D_0})| \le 1$.  In fact, an explicit relationship is the following.  \ Take a chamber $D \subset \ScptB_{N}$ (possibly equal to $C_{0}$ or $C_{1}$), and consider two possibilities depending if there is or there is no apartment containing the three chambers $C_{0}$, $C_{1}$, and $D$.

\medskip

\noindent {\sc{Case}} \ $\nexists$ apartment $\ScptA$.  Let $F$ denote the common face of $C_{0}$ and $C_{1}$.  Here, the relationships between the convex closures $C(D,F)$, $C(D,C_{0})$, and $C(D,C_{1})$ is 
$$
C(D,C_{0}) \ = \ C(D,F) \ \cup \ C_{0} \quad {\text{\rm{and}}} \quad  C(D,C_{1}) \ = \ C(D,F) \ \cup \ C_{1} \ . 
$$ 
\noindent In particular, there exists $g \in \Gk$ which fixes $C(D,F)$ and takes $C_{0}$ to $C_{1}$, and therefore $\myht_{C_{0}}(D) = \myht_{C_{1}}(D)$.  

\medskip

\noindent {\sc{Case}} \ $\exists$ apartment $\ScptA$. \ \ In the apartment $\ScptA$, let $H$ denote the (refined) affine root hyperplane which separates the (adjacent) chambers $C_{0}$ and $C_{1}$.  If $D$ and $C_{0}$ (resp.~$C_{1}$) are on the same side of $H$, then 
$\myht_{C_{1}}(D) =  \myht_{C_{0}}(D) + 1$  (resp.~$\myht_{C_{0}}(D) =  \myht_{C_{1}}(D) + 1$).

\medskip

\noindent  We deduce, from this explicit height relationship, the relationship between the outward faces of $D$ with respect to $C_{0}$ and $C_{1}$ is the following:

\smallskip

\begin{itemize}
\item[$\bullet$] \ If $\nexists$ apartment $\ScptA$, then the outward facing faces of $D$ with respect to $C_{0}$ is the same as  with respect to $C_{1}$.

\smallskip

\item[$\bullet$] \ If $\exists$ apartment $\ScptA$, and $D$ has a face $E$ in the (refined) affine root hyperplane $H$, then $E$ is outward for $C_{0}$ (resp.~$C_{0}) $ when $D$ and $C_{0}$ (resp.~$C_{1}$) are on the same side of $H$.  If $D$ does not have a face on $H$, then the outward facing faces of $D$ with respect to $C_{0}$ is the sames as with respect to $C_{1}$.
\end{itemize}

\medskip


Suppose $C_{0}$ and $C_{1}$ are adjacent chambers with common face $F$, and $F' \ne F$ is any other face.

\begin{itemize} 
\item[$\bullet$] \ If $F$ and $F'$ are not aligned, then ${\ScptC}_{C_{0}}(F') =  {\ScptC}_{C_{1}}(F')$.  To see this, we consider $C(F,F')$.  That $F$ and $F'$ are not aligned means the maximal dimension facet $Y$ (denoted as $D_{F'}$ in Lemma \ref{facet-prelim-ms}) in $C(F,F')$ containing $F'$ is a chamber, and $F$ is a outward face $Y$.  Since $C(C_{0},F') \supset C(F.F')$ as well as $C(C_{1},F') \supset C(F.F')$, it follows  $Y$ is ${\ScptC}_{C_{0}}(F')$ and ${\ScptC}_{C_{1}}(F')$.

\smallskip

\item[$\bullet$] \ If $F$ and $F'$ are aligned there exists an apartment ${\ScptA}'$ containing $F$, $F'$, ${\ScptC}_{C_{0}}(F')$, and  ${\ScptC}_{C_{1}}(F')$, and in ${\ScptA}'$ the reflection across the affine hyperplane generated by $F$ swaps    ${\ScptC}_{C_{0}}(F')$, and  ${\ScptC}_{C_{1}}(F')$.  To see this, we note that $\dim (C(F,F'))$ is that of a face.  Pick an apartment $\ScptA''$ containing $C_{0}$ and $F$. We use $F \subset C_{0}$, to say $\ScptA''$ contains $C(F,F')$, and the later generates a hyperplane in $\ScptA''$.  The Iwahori subgroup $\Gk_{C_{0}}$ acts transitively on the apartments containing $C(C_{0},F')$.  There exists $h \in \Gk_{C_{0},0}$ fixing $C(C_{0},F')$ and moving $C_{1}$ to $\ScptA''$.  The apartment $\ScptA' := h^{-1}.\ScptA''$ satisfies the assertion.

\end{itemize}

\medskip  We are now ready to prove the Proposition. \ Consider the intersection 
\begin{equation}\label{intersection-convolve-ms}
{\text{\rm{Int}}}(m) \ := \ \myBall(C_{0},m) \ \cap \ \myBall(C_{1},m) \ .
\end{equation}

\noindent For $a \in \{ 0, 1\}$, the Euler-Poincar{\'{e}} sum of the idempotents $e^{\mathcal C}_{K}$ over the facets of the ball $\myBall (C_{a},m)$ equals:
$$
{\underset {K \subset {\text{\rm{Int}}}(m)} \sum } (-1)^{\dim (K)} e^{\mathcal C}_{K} \ \ + {\underset 
{\text{\rm{\tiny $D \subset \Big( \myBall (C_{a},m) \backslash {\text{\rm{Int}}}(m) \Big)$}}} \sum } \ \ \ {\underset {{\mathcal E}_{C_{a}}(D) \subset K \subset D} \sum } (-1)^{\dim (K)}  e^{\mathcal C}_{K} \ \ .
$$

\noindent So the difference of the Euler-Poincar{\'{e}} sums over $\myBall (C_{0},m)$ and $\myBall (C_{1},m)$ is

\begin{equation}\label{difference-convolve-ms} 
\aligned
{\underset 
{\text{\rm{\tiny $D \subset \Big( \myBall (C_{0},m) \backslash {\text{\rm{Int}}}((m)) \Big)$}}} \sum } \ &\ \ {\underset {{\mathcal E}_{C_{0}}(D) \subset K \subset D} \sum } (-1)^{\dim (K)}  e^{\mathcal C}_{K} \\
&\ \ - \ \ 
{\underset 
{\text{\rm{\tiny $D \subset \Big( \myBall (C_{1},m) \backslash {\text{\rm{Int}}}((m)) \Big)$}}} \sum } \ \ \ {\underset {{\mathcal E}_{C_{1}}(D) \subset K \subset D} \sum } (-1)^{\dim (K)}  e^{\mathcal C}_{K} \ \ .
\endaligned
\end{equation}

\noindent Suppose $m \ge 1$ and $D \subset ( \myBall (C_{0},m) \cup \myBall (C_{1},m) )$.

\begin{itemize}

\item[$\bullet$] {\sc{Case}} \ ${\nexists}$ apartment $\ScptA$ containing $C_{0}$, $C_{1}$ and $D$.  \ Here $D \subset {\text{\rm{Int}}}(m)$ and therefore it does not occur as a summation index value in \eqref{difference-convolve-ms} . 

\smallskip

\item[$\bullet$] {\sc{Case}} \ ${\exists}$ apartment $\ScptA$ containing $C_{0}$, $C_{1}$ and $D$. \ Here, $D$ occurs as a summation index value in the first (resp.~second) term of \eqref{difference-convolve-ms} when $\myht_{C_{0}}(D) = m$ and $\myht_{C_{1}}(D) = (m-1)$ (resp.~$\myht_{C_{1}}(D) = m$ and $\myht_{C_{0}}(D) = (m-1)$, and the Euler-Poincar{\'{e}} sum 
$$
f \ \star \ \Big(  {\underset {{\mathcal E}_{C_{0}}(D) \subset K \subset D} \sum } (-1)^{\dim (K)}  e^{\mathcal C}_{K} \ \Big) \qquad {\text{\rm{(resp. }}} f \ \star \ \Big(  {\underset {{\mathcal E}_{C_{1}}(D) \subset K \subset D} \sum } (-1)^{\dim (K)}  e^{\mathcal C}_{K} \ \Big) {\text{\rm{ ) }}} 
$$
\noindent vanishes for $m$ sufficiently large.

\end{itemize}

Given the above, we deduce that for $m$ sufficiently large that:

$$
\aligned
f \ \star \ \big( {\underset {K \, \subset \, \myBall (C_{0},m) } \sum} \ (-1)^{\dim ( K )} \, e^{\mathcal C}_{K} \big) \ &= \ f \ \star \ \big( {\underset {K \, \subset \, {\text{\rm{Int}}} (m) } \sum} \ (-1)^{\dim ( K )} \, e^{\mathcal C}_{K} \ \big) \ \\
&= \ f \ \star \ \big( {\underset {K \, \subset \, \myBall (C_{1},m) } \sum} \ (-1)^{\dim ( K )} \, e^{\mathcal C}_{K} \ \big) \ .
\endaligned
$$

\end{proof}

\medskip

\begin{cor} \label{independence-b-convolve-ms}  Under the assumptions of subsection \eqref{a-convolve-ms}, suppose $C_{0}$ and $C_{1}$ are two chambers of $\ScptB_{N}$, and $f \in \Cic (\Gk )$.  Then, for $m$ sufficiently large:  
$$
f \ \star \ \big( {\underset {K \, \subset \, \myBall (C_{0},m) } \sum} \ (-1)^{\dim ( K )} \, e^{\mathcal C}_{K} \big) \ = \ f \ \star \ \big( {\underset {K \, \subset \, \myBall (C_{1},m) } \sum} \ (-1)^{\dim ( K )} \, e^{\mathcal C}_{K} \ \big) \ .
$$
\end{cor}
\begin{proof}  \ \ We apply Proposition \ref{independence-a-convolve-ms} to a sequence of chambers $C_{0} = D_{0} \, , \, D_{1} , \dots , , \, D_{n} = C_{1}$ in which has the property that $D_{i}$ and $D_{i+1}$ are adjacent. 
\end{proof}

\bigskip

\noindent We combine the stabilization Corollary \eqref{stabilization-convolve-ms} and the independence Corollary \eqref{independence-b-convolve-ms} to deduce the existence of a distribution ${\mydistD}^{\mathcal C}$, which we write descriptively as: 

\begin{equation}\label{distribution-convolve-ms}
{\mydistD}^{\mathcal C} \ = \ \big( {\underset {F \, \subset \, \ScptB_{N} } \sum} \ (-1)^{\dim ( F )} \, e^{\mathcal C}_{F} \ \big) \ ,
\end{equation}
\noindent so that for any $f \in \Cic (\Gk )$ and any chamber $C_{0}$, there is an integer $N(f,C_{0})$ so that 

\begin{equation}
f \ \star \ D^{\mathcal C} \ = \ f \ \star \ \big( {\underset {K \, \subset \, \myBall (C_{0},m) } \sum} \ (-1)^{\dim ( K )} \, e^{\mathcal C}_{K} \big) \qquad {\text{\rm{for $m \ge N(f,C_{0})$}}} . 
\end{equation}

\bigskip


\begin{thm}\label{main-convolve-ms} Suppose $\mk$ is a non-archimedean local field and $\mG$ is an absolutely quasisimple linear algebraic group defined over $\mk$.  Let $\Gk = \mG (\mk )$ be the group of $\mk$-rational points, and $\ScptB = \ScptB (\Gk )$ the Bruhat--Tits building of $\Gk$.  For $N \in \bN$, let $\ScptB_{N}$ be the refined building, and suppose ${\mathcal C} = \{ (F, \chi ) \}$ is the equivalence class of a cuspidal associate pair of depth $r \in {\frac{1}{N}}{\bN_{+}}$.  Let  $\{ e^{\mathcal C}_{F} \}$ be the $\Gk$-equivariant system of idempotents defined in \eqref{idempotent-pontryagin-ms}.   Then, the distribution 
$$
{\mydistD}^{\mathcal C}
\, = \, {\underset {E \subset \ScptB_{N}} \sum} \ (-1)^{\dim (E)} \, e^{\mathcal C}_{E}
$$
\noindent of \eqref{distribution-convolve-ms} satisfies the following:
\smallskip
\begin{itemize} 
\item[(i)] ${\mydistD}^{\mathcal C}$ is $\Gk$-invariant.
\smallskip
\item[(ii)] ${\mydistD}^{\mathcal C}$ is essentially compact.
\smallskip
\item[(iii)]  For a facet $K \subset \ScptB_{N}$, define
$$
e^{r}_{K} \ := \ {\mfrac{1}{\meas (\Gk_{K,r^{+}} )}} \ 1_{\Gk_{K,r^{+}}} \ \ .
$$

\noindent Then,
$$
{\mydistD}^{\mathcal C} \ \star \ e^{r}_{K} \ = \ e^{\mathcal C}_{K} \ .
$$

\smallskip

\item[(iii.1)] If {\,}$(K, \xi) \notin {\mathcal C}$, then {\,}${\mydistD}^{\mathcal C} \star e_{K,\xi} \, = \, 0$.
\smallskip
\item[(iii.2)] If {\,}$(K, \xi) \in {\mathcal C}$, then {\,}${\mydistD}^{\mathcal C} \star e_{K,\xi} \, = \, e_{K,\xi}$.

\medskip

\item[(iv)] If ${\mathcal D} \, = \, \{ (F', \chi' ) \} \, \neq \, {\mathcal C}$ is another cuspidal associate pair of depth $r \in {\frac{1}{N}}{\bN_{+}}$, then 
$$
{\mydistD}^{\mathcal C} \ \star \ {\mydistD}^{{\mathcal D}} \ = \ 0 \quad .
$$
\smallskip

\item[(v.1)]  The depth  $\le r$ projector ${P}_{\le r} \, := \, \big( {\underset {F \, \subset \, \ScptB_{N} } \sum} \ (-1)^{\dim ( F )} \, e_{\Gk_{F,r^{+}}} \big)$ is decomposed as
$$
{P}_{\le r} \ =  \ {\underset {\text{\rm{${\mathcal C}$ of depth $r$}}} \sum } {\mydistD}^{\mathcal C} \ .
$$

\item[(v.2)] $\mydistD^{\mathcal C}$ is idempotent.

\end{itemize} 
\end{thm}

\bigskip

As a preparation to the proof of Theorem \ref{main-convolve-ms}, we state and prove:

\begin{lemma}\label{lemma-b-convolve-ms} \quad Fix a base chamber {\,}$C_{0} \subset \ScptB (\Gk)_{N}$, and a facet $K \subset C_{0}$.  For any chamber $D \neq C_{0}$:
$$
\dim ( C({\mathcal E}_{C_{0}} (D) , K ) ) \ > \ \dim ( {\mathcal E}_{C_{0}} (D)) \ . 
$$
\end{lemma}

\begin{proof} \ \ Since {\,}$\dim ( C({\mathcal E}_{C_{0}} (D), K ) )  \ge \dim ( {\mathcal E}_{C_{0}} (D) )$, if suffices to show {\,}$\dim ( C({\mathcal E}_{C_{0}} (D), K ) ) = \dim ( {\mathcal E}_{C_{0}} (D) )$ leads to a contradiction.   Recall, 
$$
\myAff ( {\mathcal E}_{C_{0}} (D) ) \ = \ {\underset {F \in c_{C_{0}}(D)} \bigcap } \myAff (F) 
$$ 
\noindent is the affine subspace generated by ${\mathcal E}_{C_{0}} (D)$.  If $\dim (C({\mathcal E}_{C_{0}} (D) , K ) ) = \dim ( {\mathcal E}_{C_{0}} (D))$, then the facet $K$ must lie in $\myAff ( {\mathcal E}_{C_{0}} (D) )$.  \ In opposition to the facet ${\mathcal E}_{C_{0}} (D)$, set  

$$
\aligned
{\mathcal F}_{C_{0}} \, :&= \, {\underset {F \in p_{C_{0}}(D)} \bigcap } \ F \ \ \ {\text{\rm{(the facet of $D$ opposite to ${\mathcal E}_{C_{0}} (D)$), and}}} \\
\myAff ( {\mathcal F}_{C_{0}} (D) ) \, &= \, {\underset {F \in p_{C_{0}}(D)} \bigcap } \myAff (F) \ \ \ {\text{\rm{(the affine subspace generated by ${\mathcal F}_{C_{0}} (D)$)}}} \ .
\endaligned
$$

\noindent That ${\mathcal E}_{C_{0}} (D)$ and ${\mathcal F}_{C_{0}} (D)$ are opposite facets of $D$  means the two affine subspaces  $\myAff ( {\mathcal E}_{C_{0}} (D) )$ and $\myAff ( {\mathcal F}_{C_{0}} (D) )$ have empty intersection.   By definition, each face $F \subset D$, is contained in the zero hyperplane $H_{\pm \psi_{F}}$ of a pair of virtual affine root.  We choose the sign so that $\mygrad (\psi_{F} ) = \alpha_{F}$, i.e., the outward/inward direction of the face. Then, 
$$
K \, \subset \, \Big( \, {\underset {F \in p_{C_{0}}(D)} \bigcap } H_{\psi_{F} \le 0}\, \Big) \ \ \ {\text{\rm{and}}} \ \ \  \myAff ({\mathcal F}_{C_{0}}(D)) \, \subset \, \Big( \, {\underset {F \in c_{C_{0}}(D)} \bigcap } H_{\psi_{F} \ge 0} \, \Big) \ .
$$  

\noindent But, the above two intersections are cones whose intersection is empty-- a contradiction to the assumption $K \subset \myAff ({\mathcal F}_{C_{0}}(D))$.  \ Thus, the lemma is proved.


\end{proof}


\noindent{\it Proof of Theorem \ref{main-convolve-ms}} \quad Assertion (i) on $\Gk$-invariance of $D^{\mathcal C}$ is a consequence of the independence of base chamber Corollary \eqref{independence-b-convolve-ms}.  

\medskip

Assertion (ii) that $D^{\mathcal D}$ is an essentially compact distribution is the stabilization Corollary \eqref{stabilization-convolve-ms}.

\medskip

To prove assertion (iii), we use the independence of base chamber Corollary \eqref{independence-b-convolve-ms} to assume that the facet $K$ is contained in $C_{0}$.  Then,

$$
\aligned
e^{r}_{K} \ \star \ &\big( \, {\underset {J \subset C_{0}} \sum } (-1)^{\dim (J)} \, e^{\mathcal C}_{J} \, \big) \ = \ {\underset {J \subset C_{0}} \sum } \ (-1)^{\dim (J)} \ e^{r}_{K} \ \star \ e^{\mathcal C}_{J} \\
& \quad  = \ {\underset {J \subset C_{0}} \sum } \ (-1)^{\dim (J)} \ e^{r}_{K} \ \star \ ( e^{r}_{J} \, \star \, e^{\mathcal C}_{J} ) \  =  \ {\underset {J \subset C_{0}} \sum } \ (-1)^{\dim (J)} \ ( \, e^{r}_{K} \ \star \  e^{r}_{J} \, ) \, \star \, e^{\mathcal C}_{J}  \\
& \quad = \ {\underset {J \subset C_{0}} \sum } \ (-1)^{\dim (J)} \ e^{r}_{C(K,J)} \, \star \, e^{\mathcal C}_{J} \  = \ {\underset {J \subset C_{0}} \sum } \ (-1)^{\dim (J)} \ e^{\mathcal C}_{C(K,J)}  \ = \ e^{\mathcal C}_{K} \ \ .\\
\endaligned 
$$

\smallskip

\noindent  Thus, a sufficient condition to prove statement (iii) is that if $D$ is a chamber not equal to $C_{0}$, it is the case that the convolution 
$$
e^{r}_{K} \ \star \ \big( {\underset {{\mathcal E}_{C_{0}}(D) \, \subset J \, \subset D} \sum} \ (-1)^{\dim (J)} \, e^{\mathcal C}_{J} \ \big) \  \ {\text{\rm{vanishes}}}{\,}.
$$
\noindent To see this,  let $L$ be the unique facet of $C(K,{\mathcal E}_{C_{0}}(D))$ which contains ${\mathcal E}_{C_{0}}(D)$ and is of maximal dimension.  we calculate:

$$
\aligned
e^{r}_{K} \ \star \ &\big( {\underset {{\mathcal E}(D) \, \subset J \, \subset D} \sum} \ (-1)^{\dim (J)} \, e^{\mathcal C}_{J} \ \big) \ = \ {\underset {{\mathcal E}_{C_{0}}(D) \, \subset J \, \subset D} \sum} \ (-1)^{\dim (J)} \, e^{r}_{K} \, \star \, e^{\mathcal C}_{J} \\
& \quad  = \ {\underset {{\mathcal E}_{C_{0}}(D) \, \subset J \, \subset D} \sum} \ (-1)^{\dim (J)} \, e^{r}_{K} \, \star \,  e^{r}_{J} \, \star \, e^{\mathcal C}_{J} \ = \ {\underset {{\mathcal E}_{C_{0}}(D) \, \subset J \, \subset D} \sum} \ (-1)^{\dim (J)} \, e^{r}_{K} \, \star \, e^{r}_{L} \, \star \,  e^{r}_{J} \, \star \, e^{\mathcal C}_{J} \\
& \quad = \ {\underset {{\mathcal E}_{C_{0}}(D) \, \subset J \, \subset D} \sum} \ (-1)^{\dim (J)} \, e^{r}_{K} \, \star \, e^{r}_{C(L,J)} \, \star \, e^{\mathcal C}_{J} \ . \\
\endaligned
$$

\noindent The summation of the last line vanishes since ${\mathcal E}_{C_{0}}(D) \subsetneq L$ creates repetition in the summation.  \ Assertions (iii.1) and (iii.2) are immediate consequences of (iii).

\medskip

To complete the proof of the Theorem, we note assertion (iv) is a consequence of (iii), assertion (v.1) is obvious, and (v.2) follows from the orthogonality relation (iv) and (v.1). 

\hfill \qed


\vskip 0.70in 
 
\section{Resolutions}\label{resolutions-ms}

\medskip

\subsection{Review of work of Schneider--Stuhler and Bestvina--Savin} \ 

\medskip

We assume the hypotheses of section \eqref{a-convolve-ms}.  For a facet $K \subset \ScptB_{N}$, set 
\begin{equation}\label{bkv-idempotent-resolutions-ms}
e^{r}_{K} \ = \ {\frac{1}{\meas (\Gk_{K,r^{+}} )}} \ 1_{\Gk_{K,r^{+}}} \ \ .
\end{equation}

\noindent If $(\pi , V_{\pi})$ is a smooth representation of $\Gk$, note that 
$$
\pi ( \, e^{r}_{K} \, ) (V_{\pi}) \ = \ V^{\Gk_{K,r^{+}}}_{\pi} \ \ .
$$
 
\noindent Let ${\mathcal K} = \myOrb_{\Gk}(K)$ be the $\Gk$-orbit of the facet $K$.  Define the dimension of the orbit ${\mathcal K}$ to be the dimension of the facet $K$.  The vector space 
\begin{equation}\label{defn-a-resolutions-ms}
W^{r , {\mathcal K}}_{\pi } \ := \ \{ \ (J , v) \ | \ J \in {\mathcal K} \ , \ v_{J} \, \in \, \pi ( \, e^{r}_{J} \, ) (V_{\pi} ) \ {\text{\rm{ , and $v_{J} = 0$ for almost all $J$}}} \ \}
\end{equation}
\noindent is a smooth representation of $\Gk$ with the action 
$$
g \, . \, (J, v_{J}) \ := \ ( \,  g.J \, , \, \pi (g) (v_{J}) \, ) \ \ .
$$  
\noindent It is equivalent to the smooth representation
\begin{equation}\label{induced-resolutions-ms}
\big( \ \mycInd{\Gk}{\Gk_{K,r^{+}}} ({\text{\rm{triv}}}) \ \big) \ \otimes_{\bC} \ V^{\Gk_{K,r^{+}}}_{\pi} \ , 
\end{equation}
\noindent and therefore it is a projective $\Gk$-module.

\medskip

\noindent For $k \in \{ 1 , 2 , \dots , (\ell = {\text{\rm{rank}}} (\Gk ) ) \}$, set 

\begin{equation}\label{defn-b-resolutions-ms}
W^{r,k}_{\pi} \ := \ {\underset 
{\text{\rm{\tiny $\begin{matrix} {\mathcal K} \\ {\dim ({\mathcal K}) = k} \end{matrix}$}}} \bigoplus} W^{r , {\mathcal K}}_{\pi } \ .
\end{equation}

\medskip 

A key property of the idempotents $e^{r}_{K}$'s is 
$$
{\text{\rm{$\forall${\,} facet {\,}$E$, and subfacet {\,}$F \subset \partial (E)$,}}} \ \ {\text{\rm{it is the that case}}} \ \ \pi (e^{r}_{F}) (V_{\pi}) \subset \pi (e^{r}_{E}) (V_{\pi}) \ .
$$

\noindent This means, the boundary map
\begin{equation}\label{boundary-resolutions-ms}
\aligned
\partial \ : \ W^{r,{\mathcal K}}_{\pi} \ &\xrightarrow{\hskip 0.50in} \ W^{r,( \dim ({\mathcal K} )- 1)}_{\pi} \\
\partial \, \big( \, \oplus \ (E,v_{E}) \, \big) \ &\xrightarrow{\hskip 0.50in} \ \oplus \ ( \, \partial (E) , v_E \, ) \qquad {\text{\rm{(same $v_{E}$'s on both sides)}}} ;
\endaligned
\end{equation}
\noindent is defined, and a $\Gk$-map.  Consequently, there is a boundary $\Gk$-map
$$
W^{r,k}_{\pi} \ \xrightarrow{\ \ \partial \ \ } \ W^{r, (k-1)}_{\pi} \ \ 
k \in \{ 1 , 2 , \dots , \ell \} \ \ .
$$

\noindent Let 
$$
\aligned
W^{r,0}_{\pi} \qquad \ &\xrightarrow{\hskip 0.40in} \ V_{\pi} \\
{\underset { {\text{\rm{ \tiny $\begin{matrix} {\text{\rm{$x_0$ vertex}}} \\
{\text{\rm{(dim 0)}}} \end{matrix}$ }}} } \bigoplus } ( x_{0} , v_{x_{0}}) \ &\xrightarrow{\hskip 0.40in} \ \sum v_{x_{0}} \ \ , 
\endaligned
$$  
\noindent be the augmentation map (a $\Gk$-map).

\bigskip

\begin{thm}\label{ssbs-resolutions-ms} \quad  {\text{\rm{({\bf{Schneider--Stuhler, Bestvina--Savin}})}}} \quad 

\noindent Under the assumptions of section \eqref{a-convolve-ms}, let $\ell = \dim (\ScptB (\Gk ))$, and assume a (smooth) $\Gk$-module $V_{\pi}$ is generated by the subspaces $\pi (e^{r}_{x_{0}}) (V_{\pi})$ as $x_0$ runs over the vertices of $\ScptB (\Gk )_{N}$, i.e., the augmentation map is surjective.  Then 
$$
0 \xrightarrow{\hskip 0.40in} W^{r,\ell}_{\pi} \xrightarrow{\hskip 0.20in \partial \hskip 0.20in}  W^{r,(\ell -1)}_{\pi} \xrightarrow{\hskip 0.20in \partial \hskip 0.20in}  \cdots   \xrightarrow{\hskip 0.20in \partial \hskip 0.20in} W^{r,1}_{\pi} \xrightarrow{\hskip 0.20in \partial \hskip 0.20in} W^{r,0}_{\pi}  \ \xrightarrow{\hskip 0.40in} V_{\pi}
$$ 
\noindent is a (projective) resolution of $V_{\pi}$.
\end{thm}

\medskip

We remark that Schneider--Stuhler formulated and proved Theorem \ref{ssbs-resolutions-ms} for $N=1$ in [{\reSSa}].  Bestvina--Savin proved Theorem \ref{ssbs-resolutions-ms} for general $N$ in [{\reBS}].

\bigskip

Suppose $(F , \chi )$ is a cuspidal pair of depth $r \in {\frac{1}{N}}{\bN_{+}}$, and let ${\mathcal C} = \{ (F', {\chi}')$ the equivalence class of pairs associate to $(F , \chi )$.   \ Let ${\mydistD}^{\mathcal C}$ be the Bernstein center idempotent of Theorem \ref{main-convolve-ms}.  Application of ${\mydistD}^{\mathcal C}$ to the resolution of Theorem \ref{ssbs-resolutions-ms} gives a resolution:

$$
0 \xrightarrow{\hskip 0.20in} {\mydistD}^{\mathcal C} \, . \, W^{r,\ell}_{\pi} \xrightarrow{\hskip 0.10in \partial \hskip 0.10in}  {\mydistD}^{\mathcal C} . \, W^{r,(\ell -1)}_{\pi} \xrightarrow{\hskip 0.10in \partial \hskip 0.10in}  \cdots   \xrightarrow{\hskip 0.10in \partial \hskip 0.10in} {\mydistD}^{\mathcal C} . \, W^{r,1}_{\pi} \xrightarrow{\hskip 0.10in \partial \hskip 0.10in} {\mydistD}^{\mathcal C} . \,  W^{r,0}_{\pi}  \ \xrightarrow{\hskip 0.20in} {\mydistD}^{\mathcal C} \, . \, V_{\pi} \ .
$$ 

\noindent Via the definitions \eqref{defn-a-resolutions-ms} and \eqref{defn-b-resolutions-ms}, and statement (iii) of Theorem \ref{main-convolve-ms}, the $\Gk$-module {\,}${\mydistD}^{\mathcal C} . W^{r,k}_{\pi}${\,} has the description 

\smallskip

$$
\aligned
{\mydistD}^{\mathcal C} . \, W^{r,k}_{\pi} \ &= \  {\underset 
{\text{\rm{\tiny $\begin{matrix} {\mathcal K} \\ {\dim ({\mathcal K}) = k} \end{matrix}$}}} \bigoplus} {\mydistD}^{\mathcal C} . \, W^{r , {\mathcal K}}_{\pi } \ \\
&= \  {\underset 
{\text{\rm{\tiny $\begin{matrix} {\mathcal K} \\ {\dim ({\mathcal K}) = k} \end{matrix}$}}} \bigoplus}  
{\text{\rm{\small{ $\begin{matrix} 
\{ \ (J , v) \ | \ J \in {\mathcal K} \ , \ v_{J} \, \in \, \pi ( \, e^{\mathcal C}_{J} \, ) (V_{\pi} ) \ \ {\begin{matrix} \ \\ \ \end{matrix}} \ . \\
{\text{\rm{and $v_{J} = 0$ for almost all $J$}}} \ \}
\end{matrix}$ }}}}
\endaligned
$$

\medskip

We define 
\begin{equation}\label{defn-c-resolutions-ms}
\aligned
W^{{\mathcal C},k}_{\pi} \ :&= \  {\underset 
{\text{\rm{\tiny $\begin{matrix} {\mathcal K} \\ {\dim ({\mathcal K}) = k} \end{matrix}$}}} \bigoplus}  
{\text{\rm{\small{ $\begin{matrix} 
\{ \ (J , v) \ | \ J \in {\mathcal K} \ , \ v_{J} \, \in \, \pi ( \, e^{\mathcal C}_{J} \, ) (V_{\pi} ) \ \ {\begin{matrix} \ \\ \ \end{matrix}} \ . \\
{\text{\rm{and $v_{J} = 0$ for almost all $J$}}} \ \} \ \ ,
\end{matrix}$ }}}}
\endaligned
\end{equation}

\medskip

\noindent and note that $\dim (W^{{\mathcal C},k}_{\pi}) = 0$ unless $k \le \dim (F)$.  Then, in summary: 

\begin{prop}\label{new-resolutions-ms} \quad  Under the assumptions of section \eqref{a-convolve-ms}, and ${\mathcal C} = \{ (F,\chi ) \}$ the equivalence class of a cuspidal pair $(F,\chi)$ of depth $r$, let $L = dim (F)$.   Assume a $\Gk$-module $V_{\pi}$ is generated by the subspaces $\pi (e^{\mathcal C}_{x_{0}}) (V_{\pi})$ as $x_0$ runs over the vertices of $\ScptB (\Gk )_{N}$, i.e., the augmentation map is surjective.  Then, 
$$
0 \xrightarrow{\hskip 0.40in} W^{{\mathcal C},L}_{\pi} \xrightarrow{\hskip 0.20in \partial \hskip 0.20in}  W^{{\mathcal C},(L -1)}_{\pi} \xrightarrow{\hskip 0.20in \partial \hskip 0.20in}  \cdots   \xrightarrow{\hskip 0.20in \partial \hskip 0.20in} W^{{\mathcal C},1}_{\pi} \xrightarrow{\hskip 0.20in \partial \hskip 0.20in} W^{{\mathcal C},0}_{\pi}  \ \xrightarrow{\hskip 0.40in} V_{\pi}
$$ 
\noindent is a (projective) resolution of $V_{\pi}$.
\end{prop}

\medskip

\begin{cor}  Under the assumptions of section \eqref{a-convolve-ms}, suppose $(\pi , V_{\pi})$ is an irreducible smooth representation of $\Gk$ and {\,}$(F, \chi )${\,} is a cuspidal pair with $F \subset \ScptB (\Gk )_{N}$, so that {\,}$V^{\chi}_{\pi} \ne \{ 0 \}$.  Then, $V_{\pi}$ has cohomological dimension $\le \dim (F)$.
\end{cor}

\begin{proof} \quad Let ${\mathcal C}$ denote the equivalence class of the cuspidal pair $(F, \chi)$.  The hypothesis {\,}$V^{\chi}_{\pi} \ne \{ 0 \}$ means $\pi (e^{\mathcal C}_{x_{0}} ) (V_{\pi} ) \ne \{ 0 \}$ for any vertex $x_{0}$ of $F$, and so by the irreducibility hypothesis on $V_{\pi}$, it is generated by $\pi (e^{\mathcal C}_{x_{0}} ) (V_{\pi} ) $.  The resolution of $V_{\pi}$ provided by ${\mathcal C}$ has length $L$; thus $V_{\pi}$ has cohomological dimension $\le \dim (F)$.

\end{proof}
 
\vskip 0.70in 
 

\section{Nonsplit groups}\label{nonsplit}

\medskip

In this section, we explain the minor modifications needed in the proofs under our earlier assumption that $\mk$-defined group $\mG$ is split quasisimple to when it is connected, absolutely quasisimple (possibly non-split).   We assume $\mG$ is connected, absolutely quasisimple.  Set $\Gk = \mG (\mk)$.    We follow [{\reBCM}:{\S}{6}]. \ (We correct here a misstatement in the displayed line following (6.1.2) of [{\reBCM}], namely that the constant parts of affine roots do not run over $\bZ$ but rather ${\frac{1}{R}} \bZ$  for some positive integer $R$).

\medskip

Let $\mk^{un}$ be the maximal unramified extension of $\mk$, and let $\myGal (\mk^{un}/\mk)$ denote the Galois group.  Since $\mG$ is assumed to be absolutely quasisimple, the Bruhat--Tits building $\ScptB (\mG (\mk^{un} ))$ is a simplicial complex, with simplicial actions by $\mG (\mk^{un})$ and $\myGal (\mk^{un} /\mk )$.   The building $\ScptB (\Gk )$ is the $\myGal (\mk^{un}/\mk)$-fixed points of $\ScptB (\mG (\mk^{un} ))$.    

\smallskip

The refined building $\ScptB_{N}$ is also be defined in terms of a set of virtual affine roots.  Take a torus $\SamS$ defined over $\mk$ satisfying{\,}: \ (i) {\,}$\SamS$ is a maximal split $\mk^{un}$-torus, \ and \ (ii) {\,}  $\mS \, := \, \SamS^{\myGal (\mk^{un} / \mk )}$ is a maximal split $\mk$-torus.  
Set $\ScptA =  \ScptA (\SamS (\mk^{un}))^{\myGal (\mk^{un} /\mk )}$, an apartment in $\ScptB$.  The affine 
root system $\Psi$ on $\ScptA$ is the nonconstant restrictions to $\ScptA$ of affine roots in $\Psi( \SamS (\mk^{un} ))$. 
The root system $\Phi$ of $\Gk$ with respect to $\mS$ is  the set of gradients of affine roots.  
Let $x_0$ in $\ScptA$ be a special point such that $\Phi_{x_0}$, the set of gradients of affine roots vanishing at $x_0$, contains all short roots in $\Phi$. Note that the group of  automorphisms of the affine Dynkin diagram acts transitively on such special vertices (see tables in [{\reTa}:{\S}4]).
Consider $\ScptA$ a vector space with $x_0$ its origin. For every $\psi \in \Psi$ define an affine functional 
\[ 
\psi_N(x) :=\frac{1}{N} \psi(Nx). 
\] 
Let $\Psi_N$ be the set of these affine functionals. This set is independent of the choice of $x_0$.   In terms of the graphs of affine functionals in $\ScptA \times \bR$, the graphs of the functionals in $\Psi_N$ are obtained by taking two parallel graphs of affine roots in $\Psi$, and adding $(N-1)$ equally spaced parallel graphs.  Denote by $\ScptA_{N}$ the refined simplicial decomposition of $\ScptA$ by the zero loci of the virtual affine roots $\Psi_{N}$, and by $\ScptB_{N}$ the resulting refined simplicial decomposition of $\ScptB$.  Observe that if $\psi\in \Psi$ then $\psi+1\in \Psi$. Indeed, this is true for affine roots of the quasi-split groups $\mG (\mk^{un})$, by observation (see the example of odd unitary group below), and therefore for $\ScptG$ by restriction.   In particular, it is clear that $\psi_{N} +1/N \in \Psi_N$.  If $F$ is a facet of $\ScptB_{N}$, and $r \in {\frac{1}{N}} \bN_{\ge 0}$, then  
$$
\forall \ x \, , \, y \ \in \ {\text{\rm{interior of}}} \ F \ \ : \quad \Gk_{x,r} \ = \ \Gk_{y,r} \quad {\text{\rm{and}}} \quad \Gk_{x,r^{+}} \ = \ \Gk_{y,r^{+}} \ . 
$$

\noindent {\sc{Example.}} \ Assume $p\neq 2$. \ Take $\mG = \mS \mU_{2n+1}$ to be the special unitary group over a ramified quadratic extension.  Here $\ScptA=\mathbb R^{n}$  and $\Psi$ 
 has simple roots $e_1-e_2, \ldots e_{n-1}-e_n, e_n$ and $1/2-2e_1$. 
 The affine roots are $2e_i +c +1/2$, $\pm e_i\pm e_j +c/2$, $i\neq j$, 
  and $e_i + c/2$ where $c$ is an integer. The root system $\Phi$ is of the type $BC_n$. The affine chamber corresponding to the given set of simple 
  roots has two special vertices, $x_0=(0, \ldots , 0)$ and $y_0=(\frac14 , \ldots , \frac14)$. Observe that $\Phi_{x_0}$ has type $B_n$, 
  while $\Phi_{y_0}$ has type $C_n$, thus we use $x_0$ to define virtual affine roots. If $N=2$ these are 
   $2e_i +c +1/4$, $\pm e_i\pm e_j +c/2$, $i\neq j$, 
  and $e_i + c/2$ where $c$ is half-integer.   \ In particular, we note the refined simplicial structure of $\ScptB_{2}$ is that of $\ScptB$.  This is caused by the root system $\Phi$ being nonreduced.

\smallskip

An important difference between the spilt and non-split case is that the set of affine roots may be non-reduced. An affine root $\psi\in \Psi$ is called divisible if $\psi/2$ is also an affine root (in which case the two affine roots obviously have the same zero locus).  In the context of Lemma \ref{facet-simple-roots-keyprop-ms}, we take the root $\alpha_F$, perpendicular to the face $F$ of a chamber, to be the gradient of a non-divisible root vanishing on $F$. With this convention, the formulations and proofs for the split case then hold for the non-split case too.


\vskip 0.70in 
 

\section{Acknowledgments}

\medskip

The first author thanks the University of Utah Mathematics Department for its hospitality in 2018 and 2019 when parts of this work was done.  The first author thanks Dan Barbasch, Dan Ciubotaru, Dragan Mili{\v{c}}i{\'{c}} and Fiona Murnaghan for useful conversations.

 
\vskip 0.70in 
 
\hrule

\smallskip

\hrule

\vskip 0.30in

\begin{multicols}{2}
[
\section{Index of notation}
]
\bigskip

$A_{E,F}$, $A_{F}$ \ \hfill  \eqref{decomp-2-pontryagin-ms}

\medskip

$\alpha_{F}$  \hfill  \eqref{definition-a-keyprop-ms}

\medskip

$\myBall (C_{0},m)$  \hfill  \eqref{dist-ball-convolve-ms}

\medskip

$C(E,F)$ \hfill \eqref{c-prelim-ms},  \eqref{ccc-prelim-ms}

\medskip

$c_{C_{0}}(D)$  \hfill  \eqref{child-parent-keypro-ms}

\medskip

${\ScptC}_{C_{0}}(F)$ \hfill \eqref{unique-convolve-ms}

\medskip

$D_{E}$ \hfill \eqref{facet-prelim-ms}

\medskip

${\text{\rm{dist}}}_{\ScptA}$, ${\text{\rm{dist}}}_{\Delta}$  \hfill \eqref{lemma-a-convolve-ms}

\medskip 

${\mydistD}^{\mathcal C}$, $P^{\my, {\mathcal C}}$  \hfill  \eqref{distribution-convolve-ms},  \eqref{ep-sum-intro-ms}

\medskip


$\Delta_{C}(K)$  \hfill  \eqref{min-para-ep-ms}

\medskip

$e_{J}$  \hfill  \eqref{haar-prelim-ms}

\medskip

$e_{E,A}$ \hfill  \eqref{idem-char-pontryagin-ms}

\medskip

$e^{\my, {\mathcal C}}_{D}$  \hfill \eqref{idempotent-pontryagin-ms}

\medskip

$e^{r}_{K}$  \hfill  \eqref{definition-a-keyprop-ms}

\medskip

$H_\xi$ \hfill \eqref{hyperplane-defn-prelim-ms}

\medskip

$P_{\le r}$ \hfill \eqref{depth-ler-projector-intro-ms}

\medskip

$p_{C_{0}}(D)$  \hfill  \eqref{child-parent-keypro-ms}

\medskip

$\partial$  \hfill  \eqref{boundary-resolutions-ms}

\medskip

$S(C_{0},\Delta )$  \hfill  \eqref{definition-b-keyprop-ms}

\medskip

$\mystar_{\mathcal C}(D)$  \hfill \eqref{local-star-pontryagin-ms}

\medskip

$\Phi (\mS)$, $\Phi$ \hfill \eqref{hyperplane-defn-prelim-ms}

\medskip

$\Phi_{E}$ \hfill \eqref{constant-prelim-ms}

\medskip

$\chi (D)$ \hfill \eqref{idempotent-pontryagin-ms} 

\medskip

$\Psi (\mS)$, $\Psi$ \hfill \eqref{hyperplane-defn-prelim-ms}

\medskip

$\Psi (\mS)_{N}$ \hfill \eqref{virtualn-prelim-ms}

\medskip

$V^{*}$, $\widehat{V}$ \hfill \eqref{direct-pontryagin-ms}

\medskip

$V^{r}_{E}$, $V^{r}_{E,F}$ \hfill \eqref{direct-pontryagin-ms}

\medskip

${\myVsph}(E)$  \hfill  \eqref{virtual-sphere-prelim-ms}

\medskip

$\myVsph_{C} (K)$  \hfill  \eqref{a-ep-ms}

\medskip

$W^{r,k}_{\pi}$ \hfill \eqref{resolution-b-intro-ms}, \eqref{defn-b-resolutions-ms}

\end{multicols}

\vfill
 
\vskip 0.70in 
 

\vfill
\vfill

\vfill
{\small{

{\tsc{Department of Mathematics, The Hong Kong University of Science and Technology, Clear Water Bay Road, Hong Kong, Email:{,}}}{\tt{amoy{\char'100}ust.hk}}

\vskip 0.10in
{\tsc{Department of Mathematics, University of Utah, Salt Lake City, UT 84112, USA}}

\noindent {\tsc{Email:{\,}}}{\tt{savin{\char'100}math.utah.edu}}

}}

\eject

\vfill
 
}} 
 
\end{document}